\documentclass[11pt,letterpaper]{article}

\usepackage{amssymb,amsfonts,amscd,amsthm}
\usepackage[all,arc]{xy}
\usepackage{enumerate, bbm}
\usepackage{mathrsfs}
\usepackage{tikz}
\usepackage[left=1in,top=1in,right=1in,bottom=1in]{geometry}
\usepackage{mathtools}
\usepackage{comment}
\usepackage{hyperref}
\usepackage{color}
\usepackage{graphicx}
\usepackage{enumitem} 
\graphicspath{ {TexImages/} }

\newtheorem{thm}{Theorem}[section]

\newtheorem{prop}[thm]{Proposition}
\newtheorem{lem}[thm]{Lemma}

\newtheorem{prob}[thm]{Problem}

\newtheorem{conj}[thm]{Conjecture}

\hypersetup{
	colorlinks,
	citecolor=blue,
	filecolor=blue,
	linkcolor=blue,
	urlcolor=blue,
	linktocpage
}

\setenumerate[1]{label=\thesection.\arabic*.} 
\setenumerate[2]{label*=\arabic*.} 
\setlist[enumerate]{itemsep=2ex, topsep=2ex} 
\setlist[itemize]{itemsep=2ex, topsep=2ex}

\newcommand{\R}{\mathbb{R}}

\newcommand{\N}{\mathbb{N}}

\newcommand{\E}{\mathbb{E}}

\newcommand{\al}{\alpha}
\newcommand{\be}{\beta}
\newcommand{\gam}{\gamma}

\newcommand{\sig}{\sigma}
\newcommand{\ep}{\epsilon}
\newcommand{\lam}{\lambda}

\newcommand{\Om}{\Omega}

\newcommand{\del}{\delta}
\newcommand{\Del}{\Delta}
\newcommand{\pa}{\partial}

\renewcommand{\l}{\left}
\renewcommand{\r}{\right}

\newcommand{\half}{\frac{1}{2}}
\newcommand{\quart}{\frac{1}{4}}

\newcommand{\sm}{\setminus}
\newcommand{\sub}{\subseteq}

\renewcommand{\c}[1]{\mathcal{#1}}
\renewcommand{\b}[1]{\mathbf{#1}}

\newcommand{\ol}[1]{\overline{#1}}
\newcommand{\tr}[1]{\textrm{#1}}
\newcommand{\rec}[1]{\frac{1}{#1}}
\newcommand{\f}[2]{\frac{#1}{#2}}

\newcommand{\ex}{\mathrm{ex}}
\newcommand{\HI}[1]{\c{H}(#1)}

\newcommand{\Con}[1]{\tilde{H}(n;#1)}
\begin{document}
	
	\title{\vspace{-0.5in} Relative Tur\'{a}n Problems for Uniform Hypergraphs}
	
	\author{Sam Spiro\thanks{Department of Mathematics, University of California, San Diego, 9500 Gilman Drive, La Jolla, CA 92093-0112, USA. E-mail: sspiro@ucsd.edu. This material is based upon work supported by the National Science Foundation Graduate Research Fellowship under Grant No. DGE-1650112.}\and
		Jacques Verstra\"ete\thanks{Department of Mathematics, University of California, San Diego, 9500 Gilman Drive, La Jolla, CA 92093-0112, USA. E-mail: jacques@ucsd.edu. Research supported by the National Science Foundation Awards DMS-1800332 and DMS-1952786,
			and by the Institute for Mathematical Research (FIM) of ETH Z\"urich.}}

	\maketitle
	
	\begin{abstract}
		\noindent	For two graphs $F$ and $H$, the relative Tur\'{a}n number $\mathrm{ex}(H,F)$ is the maximum number of edges in an $F$-free subgraph of $H$. Foucaud, Krivelevich, and Perarnau~\cite{FKP} and Perarnau and Reed~\cite{PR} studied these quantities as a function of the maximum degree of $H$.
		
		\medskip
		
		In this paper, we study a generalization for uniform hypergraphs.
		If $F$ is a complete $r$-partite $r$-uniform hypergraph with parts of sizes $s_1,s_2,\dots,s_r$ with each $s_{i + 1}$
		sufficiently large relative to $s_i$, then with $1/\beta = \sum_{i = 2}^r \prod_{j = 1}^{i - 1} s_j$ we prove
		that for any $r$-uniform hypergraph $H$ with maximum degree $\Delta$,
		\[
		\ex(H,F)\ge \Delta^{-\beta - o(1)} \cdot e(H).
		\]
		This is tight as $\Delta \rightarrow \infty$ up to the $o(1)$ term in the exponent, since we show there exists a $\Delta$-regular $r$-graph $H$ such that
		$\mathrm{ex}(H,F)=O(\Delta^{-\beta}) \cdot e(H)$. Similar tight results are obtained when $H$ is the random $n$-vertex $r$-graph $H_{n,p}^r$ with edge-probability $p$, extending
		results of Balogh and Samotij~\cite{BS} and Morris and Saxton~\cite{MS}.  General lower bounds for a wider class of $F$ are also obtained.
	\end{abstract}
	
	\section{Introduction}

	The {\em Tur\'{a}n number} $\ex(n,F)$ of a graph $F$ is the maximum number of edges in an $F$-free $n$-vertex graph. The Tur\'{a}n numbers are a central object of study in extremal graph theory, dating back to Mantel's Theorem~\cite{M} and Tur\'{a}n's Theorem~\cite{T}. Given a host graph $H$, we define the {\em relative Tur\'{a}n number} $\ex(H,F)$ to be the maximum number of edges in an $F$-free subgraph of $H$, and this is precisely $\ex(n,F)$ when $H = K_n$. The study of $\ex(H,F)$ for various graphs $F$ and $H$ has attracted considerable attention in the literature. One observes that if $F$ has chromatic number $k\ge 3$, then by taking a maximum $(k-1)$-partite subgraph we find for all $H$
	\[\ex(H,F)\ge \l(1-\rec{k-1}\r) \cdot e(H),\]
	which is best possible by the Erd\H{o}s-Stone Theorem, which shows $\ex(K_n,F) \sim (1 - \rec{k-1})e(K_n)$.
	
	\medskip
	
	The case $F$ is bipartite was studied at length by
	Foucaud, Krivelevich, and Perarnau~\cite{FKP}, who conjectured that if $F$ and $H$ are graphs such that $H$ has minimum degree $\del$ and maximum degree $\Del$, then $H$ has a spanning $F$-free subgraph of minimum degree $\Omega(\del \ex(\Del,F)/\Del^2)$ as $\Del \rightarrow \infty$, and more generally that this holds for any family of graphs $\c{F}$. This conjecture is true if $F$ has chromatic number $k \geq 3$, since a maximum $(k-1)$-partite subgraph of a graph of minimum degree $\del$ can be chosen to have minimum degree at least $(1 - \rec{k-1})\del$. The conjecture was proved up to a logarithmic factor for $\c{F}=\{C_3,C_4,\ldots,C_{2\ell}\}$ by Foucaud, Krivelevich, and Perarnau~\cite{FKP}, and later Perarnau and Reed~\cite{PR} proved the conjecture for this $\c{F}$ along with other cases such as all bipartite graphs $F$ of diameter at most three. The following
	conjecture appears to be at the heart of the above conjecture and of the same level of difficulty:
	
	\begin{conj} \label{conj:graphHost}
		If $F$ and $H$ are graphs such that $H$ has maximum degree $\Del$, then  as $\Del \rightarrow \infty$,
		\[ \ex(H,F) = \Omega\Bigl(\frac{\ex(\Del,F)}{\Del^2}\Bigr) \cdot e(H).\]
	\end{conj}
	
	It is generally an open problem to find such an $F$-free subgraph of $H$ when $F$ is bipartite, contains a cycle, and has diameter larger than three.  We note that one reason may suspect that Conjecture~\ref{conj:graphHost} holds is that the clique $K_{\Del+1}$ is the densest graph of maximum degree $\Del$.  This means that it should be relatively hard to delete copies of $F$ from $K_{\Del+1}$, which suggests that $\ex(K_{\Del+1},F)=\ex(\Del+1,F)$ should be relatively small compared to any other graph $H$ with maximum degree $\Del$.
	
	In this paper we generalize the relative Tur\'an number to $r$-uniform hypergraphs, which we call $r$-graphs for short.  For two $r$-graphs $H,F$ we define the {\em relative Tur\'{a}n number} $\ex(H,F)$  to be the maximum number of edges in an $F$-free subgraph of $H$. It follows from results of Katona, Nemetz, and Simonovits~\cite{KNS} that if $F$ is not $r$-partite, then for any $r$-graph $H$ we have
	$\ex(H,F) \geq (c(F) - o(1)) \cdot e(H)$, where $c(F) = \lim_{n \rightarrow \infty} \ex(n,F)/{n \choose r}$ is the {\em Tur\'{a}n density} of $F$.
	In particular, equality holds when $H = K_n^r$, the complete $r$-graph on $n$ vertices. The problem of determining $c(F)$ when $F$ is not $r$-partite is a famous open problem in extremal hypergraph theory,
	and the notorious conjecture $c(K_4^3)= 5/9$ is known as Tur\'{a}n's conjecture -- see Keevash~\cite{Keevashsurvey} for a survey of hypergraph Tur\'{a}n problems.

	\subsection{Complete $r$-partite $r$-graphs}
	
	For positive integers $2\le s_1\le \cdots \le s_r$, define $ K_{s_1,\ldots,s_r}$ to be the {\em complete $r$-partite $r$-graph}, which has vertex set $U_1\cup \cdots \cup U_r$ with $|U_i|=s_i$ for all $i$, and which has all edges of the form $\{u_1,\ldots,u_r\}$ with $u_i\in U_i$ for all $i$.
	We prove the following almost tight theorem on relative Tur\'{a}n numbers for complete $r$-partite $r$-graphs:

	\begin{thm}\label{thm:K2}
		For $r\ge 2$, let $2\le s_1\le \cdots \le s_r$ be integers and $a_i = \prod_{j < i} s_j$ for $1 \leq i \leq r$.
		\begin{enumerate}
			\item[$1$.] For any (sufficiently large) $\Del$, there exists an $r$-graph $H$ which is $\Del$-regular such that as $\Del \rightarrow \infty$,
			\begin{equation*}
			\ex(H,K_{s_1,\ldots,s_r})=O\l(\Del^{\f{-1}{\sum_{i=2}^r a_i}}\r) \cdot e(H).
			\end{equation*}
			\item[$2$.] There exist functions $f_i : \mathbb Z^+ \rightarrow \mathbb Z^+$ for $1< i \le r$ such that if $s_{i } \geq f_i(s_{i-1})$ for $1< i \le r$, then for any $r$-graph $H$ with maximum degree $\Del$, as $\Del \rightarrow \infty$,
			\begin{equation*}
			\ex(H,K_{s_1,\ldots,s_r})\ge \Del^{\f{-1}{\sum_{i=2}^r a_i}-o(1)} \cdot e(H).
			\end{equation*}
		\end{enumerate}
	\end{thm}
	We note that the second part of Theorem~\ref{thm:K2} will more or less follow from the more general results Theorems~\ref{thm:homGen} and \ref{thm:codeg} stated below.  We also note that this theorem extends results of Perarnau and Reed~\cite{PR} for relative Tur\'{a}n numbers of complete bipartite graphs $K_{s_1,s_2}$.  In this setting, they proved the second half of Theorem~\ref{thm:K2} without a $o(1)$ term, and it would be of interest to determine whether this error term can be removed for $r\ge 3$ as well. 
	
	According to Theorem \ref{thm:K2}, if $s_1,s_2,\dots,s_r$ grow sufficiently fast, then
	\begin{equation}\label{genTuranexp}
	\beta(K_{s_1,\ldots,s_r}) := \lim_{\Del \rightarrow \infty}\sup_H \frac{\log e(H)/\ex(H,K_{s_1,\ldots,s_r})}{\log \Delta} = \frac{1}{\sum_{i = 2}^r s_1s_2\cdots s_{i-1}},
	\end{equation}
	where the supremum ranges over all $H$ with maximum degree at most $\Del$.  The functions $f_i$ for $1 \leq i < r$ in Theorem \ref{thm:K2} are based on the current state of knowledge of the hypergraph Tur\'{a}n numbers\footnote{By using recent results of Pohoata and Zakharov~\cite{PZ}, it can be shown that we can take the functions to be $f_i(t)=((r-1)t^{i-1}-1)!+1$.}
	$\ex(n,K_{s_1,\ldots,s_r})$, and this condition is unnecessary if the following conjecture attributed to Erd\H{o}s~\cite{E} is true: if
	$s_1 \leq s_2 \leq \dots \leq s_r$ then
	\begin{equation}\label{completepartite}
	\ex(n,K_{s_1,\ldots,s_r})= \Theta\bigl(n^{r- \frac{1}{s_1 s_2 \cdots s_{r-1}}}\bigr).
	\end{equation}
	This conjecture was first stated explicity by Mubayi~\cite{Mubayi}, and at $r=2$ constitutes the notorious Zarankiewicz problem~\cite{Z}. When $r = 2$ and $s_2 > (s_1 - 1)!$, the conjecture was solved by Alon, Koll\'{a}r, R\'{o}nyai and Szab\'{o}~\cite{ARS,KRS}. By adapting the random polynomial constructions introduced by Bukh and Conlon~\cite{BC}, this conjecture was proved by Ma, Yuan, and Zhang~\cite{MYZ} when $s_r$ is large enough relative to $s_{r-1}$. Thus in this setting where (\ref{completepartite}) holds,
	\begin{equation}\label{Turanexp}
	\alpha(K_{s_1,\ldots,s_r}) := \lim_{n \rightarrow \infty} \frac{\log {n \choose r}/\ex(n,K_{s_1,\ldots,s_r})}{\log {n-1\choose r-1}} =  \frac{1}{(r-1)s_1s_2\cdots s_{r-1}}.
	\end{equation}
	This is strictly less than $\beta(K_{s_1,\ldots,s_r})$ when $r \geq 3$ and $\al(K_{s_1,\ldots,s_r})=\be(K_{s_1,\ldots,s_r})$ when $r=2$, as was originally proven by Perarnau and Reed~\cite{PR}.  
	
	This implies that the natural generalization of Conjecture~\ref{conj:graphHost} for $r$ graphs is false for $r\ge 3$, i.e. there exist $r$-graphs $F$ such that the clique with maximum degree $\Del$ is not an asymptotic minimizer of $\ex(H,F)$ as $H$ ranges over $r$-graphs with maximum degree $\Del$.  We note that the $H$ we use to prove the second half of Theorem~\ref{thm:K2} is a certain unbalanced complete $r$-partite $r$-graph, and a similar construction was used by Foucaud, Krivelevich and Perarnau~\cite{FKP} for a related problem where $H$ had a fixed number of edges.  
	
	\subsection{Random Hypergraphs}
	
	We recall that $H_{n,p}^r$ is the $r$-graph on $n$ vertices where each edge of $K_n^r$ is added to $H_{n,p}^r$ independently and with probability $p$ -- so $H_{n,p}^2 = G_{n,p}$.  If $(A_n)_{n \geq 1}$ is a sequence of events in a probability space, then we say $A_n$ holds \textit{asymptotically almost surely} (abbreviated a.a.s.) if $\lim_{n \rightarrow \infty} \mathbb P(A_n) = 1$.  A central conjecture of Kohayakawa, \L uczak and R\"{o}dl~\cite{KLR} was resolved independently by Conlon and Gowers~\cite{CG} and by Schacht~\cite{Sch}, which determines $\ex(G_{n,p},F) = (1 - \rec{k-1} + o(1))p{n \choose 2}$ a.a.s.\ whenever $F$ has chromatic number $k \geq 3$ and $p=\Om(n^{-1/m_2(F)})$, where $m_2(F)$ is the so-called 2-density of $F$. The case $F=C_4$ was essentially resolved by F\"uredi~\cite{F}.  This work was generalized to even cycles by Kohayakawa, Kreuter and Steger~\cite{KKS} and to complete bipartite graphs by Balogh and Samotij~\cite{BS}, and both of these cases were further refined by Morris and Saxton~\cite{MS}. Using the method of containers together with some probabilistic techniques, we prove the following theorem, which generalizes results of Balogh and Samotij~\cite{BS} and Morris and Saxton~\cite{MS} for the case $r = 2$:

	\begin{thm}\label{thm:randK}
		For $r\ge 2$, let $2\le s_1\le \cdots \le s_r$ be integers, $a_i = \prod_{j = 1}^{i - 1} s_j$ for $i=r,r+1$, and
		\[\be_1=\f{\sum_{i=1}^r s_i-r}{a_{r+1}-1} \quad \quad \mbox{ and } \quad \quad \ \be_2=\f{a_r(\sum_{i=1}^{r-1}s_i-r)+1}{(a_r-1)(a_{r+1}-1)}.\]
		If $\ex(n,K_{s_1,\ldots,s_r})= \Om(n^{r-1/a_r})$, then a.a.s.
		\[\ex(H_{n,p}^r,K_{s_1,\ldots,s_r})= \begin{cases}
		\Theta(pn^r) & n^{-r/2}\log n \leq p \leq n^{-\beta_1}, \\
		n^{r-\be_1+o(1)} & n^{-\be_1}\le p\le n^{-\be_2}(\log n)^{2a_r/(a_r-1)},\\  \Theta(p^{1-1/a_r}n^{r-1/a_r})& n^{-\be_2}(\log n)^{2a_r/(a_r-1)}\le p \le 1. \end{cases}\]
		More precisely, for $n^{-\be_1}\le p\le n^{-\be_2}(\log n)^{2a_r/(a_r-1)}$, we have a.a.s.
		\[ \Om(n^{r-\be_1})=\ex(H_{n,p}^{r},K_{s_1,\ldots,s_r})=O(n^{r-\be_1}(\log n)^2).\]
	\end{thm}
	We note that for $p<n^{-r/2}\log n$ it is easy to show that $\E[\ex(H_{n,p}^r,K_{s_1,\ldots,s_r})]=\Theta(pn^r)$, but a slightly different argument than the one we present is needed to show that the result holds a.a.s.
	
	\subsection{General Results}
	Given a family of $r$-graphs $\c{F}$, we define $\ex(n,\c{F})$ and $\ex(H,\c{F})$ to be the maximum number of edges in an $\c{F}$-free subgraph of $K_n^r$ and $H$, respectively.  The proof techniques used for Theorem~\ref{thm:K2} generalize to other families of $r$-graphs, and at its core it relies on two general results which prove effective lower bounds on $\ex(H,\c{F})$ depending on if $H$ has small or large codegrees.
	
	
	We say that an $r$-partite $r$-graph $F$ on $U_1\cup \cdots\cup U_r$ is \textit{tightly connected} if for all $i$ and distinct $u_1,u_2\in U_i$ there exist edges $e_1,e_2\in E(F)$ with $u_i\in e_i$ and $|e_1\cap e_2|=r-1$.  For example, $K_{s_1,\ldots,s_r}$ is tightly connected.
	We prove the following theorem for $r$-graphs $H$ with low codegrees.  Here the $(r-1)$-degree of a set $S\sub V(H)$ is defined to be the number of edges in $H$ containing $S$.
	
	\begin{thm}\label{thm:homGen}
		For $r\ge 2$, let $\c{F}$ be a family of tightly connected $r$-graphs with $\ex(n,\c{F})=\Om(n^{r-\gam})$ for some $\gam>0$.  If $H$ is an $r$-graph with maximum $(r-1)$-degree at most $D\ge 1$, then \[\ex(H,\c{F})=\Om(D^{-\gam}) \cdot e(H).\]
	\end{thm}
	
	Given an $r$-partite $r$-graph $F$ and $r$-partition $U_1\cup \cdots \cup U_r$, define $\pa F\big[\bigcup_{j\ne i} U_j\big]$ to be the $(r-1)$-graph on $\bigcup_{j\ne i} U_j$ with all edges of the form $e\cap \bigcup_{j\ne i} U_j$ for all $e\in E(F)$.  That is, it is the $(r-1)$-graph induced by the parts excluding $U_i$.   Define the \textit{projection family} \[\textstyle{\boldsymbol{\pi}(F)=\{\pa F\big[\bigcup_{j\ne i} U_j\big]:i\in[r],\ U_1,\ldots,U_r\tr{ is an }r\tr{-partition of }F\}}.\]  For example, $\boldsymbol{\pi}(K_{s_1,s_2,s_3})=\{K_{s_1,s_2},K_{s_1,s_3},K_{s_2,s_3}\}$. For a family of $r$-graphs $\c{F}$ we define $\boldsymbol{\pi}(\c{F})=\bigcup_{F\in \c{F}}\boldsymbol{\pi}(F)$.  The following theorem is effective for $r$-graphs $H$ with high codegrees:
	
	\begin{thm}\label{thm:codeg}
		For $r\ge 3$, let $\c{F}$ be a family of $r$-partite $r$-graphs and $\gam>0$ such that for all $\tilde{\Del}\ge 3$ and all $(r-1)$-graphs $G$ with maximum degree at most $\tilde{\Del}$,
		\[\ex(G,\boldsymbol{\pi}(\c{F}))= \Om( \tilde{\Del}^{-\gam}(\log \tilde{\Del})^{3-r})\cdot e(G).\]
		If $H$ is an $r$-partite $r$-graph  with maximum degree at most $\Del\ge 2$ such that at least half of the edges of $H$ contain an $(r-1)$-set with $(r-1)$-degree at least $D$, then
		\[\ex(H,\c{F})=\Om( \Del^{-\gam}D^{\gam}(\log \Del)^{2-r}) \cdot e(H).\]
	\end{thm}
	The logarithmic terms of Theorem~\ref{thm:codeg} are a product of its proof, and we suspect that these terms can be removed.
	
	\subsection{Tight Cycles}
	
	Theorems \ref{thm:homGen} and \ref{thm:codeg} allow us to prove effective bounds on $\ex(H,\c{F})$ for a wide variety of $r$-graphs, and for simplicity we focus on the case of
	tight cycles. For integers $r < k$, the {\em tight $k$-cycle}  $TC_{k}^r$ is the $r$-graph with vertex set $\{u_0,\ldots,u_{k - 1}\}$ consisting of the edges $\{u_i,u_{i+1},\ldots,u_{i+r-1}\}$ for $0 \le i < k$  with subscripts written modulo $k$. For instance, $TC_{r+1}^r = K_{r+1}^r$ and $TC_{k}^r$ is $r$-partite if and only if $k$ is a multiple of $r$, in which case its unique $r$-partition up to relabeling of parts has $u_j\in U_i$ whenever $j\equiv i\mod r$. For $r = 2$, the tight $k$-cycle $TC_{k}^2$ is precisely $C_{k}$, the cycle of length $k$, and a well-known conjecture of
	Erd\H{o}s and Simonovits~\cite{ES} states that for all $\ell \geq 2$,
	\begin{equation}\label{conj:es}
	\ex(n,\{C_3,C_4,\dots,C_{2\ell}\}) = \Theta(n^{1 + 1/\ell})
	\end{equation}
	as $n \rightarrow \infty$. When $r = 2$ and $\ell \in \{2,3,5\}$, (\ref{conj:es}) is true due to the existence of generalized polygons -- see Benson~\cite{Benson} for the first description in terms of extremal graph theory, Wenger~\cite{W} for an elementary presentation, and~\cite{V} for a survey. We prove the following theorem for relative Tur\'{a}n numbers of tight cycles:
	
	\begin{thm}\label{thm:tightGen}
		Let $r \geq 2$ and let $\ell \in \{2,3,5\}$. If $H$ is any $r$-graph with maximum degree $\Del$, then as $\Del \to \infty$,
		\begin{equation}\label{tight:genlower}
		\ex(H,\{TC_{r+1}^r,TC_{r+2}^r,\ldots,TC_{\ell r}^r\}) \ge \Del^{\f{-\ell+1}{\ell(r-1)}-o(1)}\cdot e(H).
		\end{equation}
	\end{thm}
	
	The inequality (\ref{tight:genlower}) generalizes the results of Foucaud, Krivelevich and Perarnau~\cite{FKP} for short cycles of even length in graphs. The proof
	of Theorem \ref{thm:tightGen} relies on an effective lower bound on the extremal function for tight cycles $TC_{\ell r}^r$ when $\ell \in \{2,3,5\}$.
	The first moment method in the random $r$-graph $H_{n,p}^r$ gives $\ex(n,TC_{\ell r}^r) = \Omega(n^{r - 1 + (r - 1)/(\ell - 1)})$ for any $\ell \geq r + 1$.
	We give a simple proof of a slight improvement as follows:
	
	\begin{thm}\label{thm:tightTuran}
		For all $r \geq 2$ and $\ell \in \{2,3,5\}$,
		\begin{equation}\label{tight:lower}
		\ex(n,\{TC_{r+1}^r,TC_{r+2}^r,\ldots,TC_{\ell r}^r\}) = \Omega(n^{r-1+1/\ell}).
		\end{equation}
	\end{thm}
	
	Upper bounds for this extremal function were left as an open problem by Conlon~\cite{C} in connection with extremal problems for cycles in hypercubes,
	and for example the current best upper bound for $TC_6^3$ is $\ex(n,TC_6^3) \leq \ex(n,K_{2,2,2}) = O(n^{11/4})$.

	\medskip

	\subsection{Organization and Notation}
	
	In Section~\ref{sec:hom} we prove Theorem~\ref{thm:homGen} using random homomorphisms.  In Section~\ref{sec:codeg} we prove Theorems~\ref{thm:codeg} and \ref{thm:tightTuran}.  We then prove our main results for general hosts Theorems~\ref{thm:K2} and \ref{thm:tightGen} in Section~\ref{sec:proofs}. In Section~\ref{sec:rand} we prove our main result for random hosts Theorem~\ref{thm:randK}.  Concluding remarks and open problems are given in Section~\ref{sec:concluding}.
	
	We gather some notation and definitions that we use throughout the text.  A set of size $k$ will be called a $k$-set.  If $H$ is an $r$-graph, then the number of edges containing a $k$-set $S=\{v_1,\ldots,v_k\}\sub V(H)$ is called the $k$-degree of $S$ and is denoted by $d_H(S)$ or $d_H(v_1,\ldots,v_k)$, and we omit the subscript wherever $H$ is understood from context. If $\chi$ is a map from vertices of $H$ and $e=\{v_1,\ldots,v_r\}\in E(H)$, we define the set  $\chi(e)=\{\chi(v_1),\ldots,\chi(v_r)\}$.  We often make use of the following basic fact due to Erd\H{o}s and Kleitman~\cite{EK}: every $r$-graph $H$ has an $r$-partite subgraph with at least $r^{-r}e(H)$ edges.  Throughout the text we omit ceilings and floors for ease of presentation.
	
	\section{Hosts with Low Codegrees : Proof of Theorem \ref{thm:homGen}}\label{sec:hom}
	
	We begin with an informal discussion of the technique for giving lower bounds on $\ex(H,F)$ when $H$ has low codegrees, which is based on techniques of Foucaud, Krivelevich and Perarnau~\cite{FKP} and Perarnau and Reed~\cite{PR} for graphs.  We will try to construct a subgraph  $H'\sub H$ that ``looks like'' another $r$-graph $J$ which is $F$-free and has many edges.  One way to try and do this is to consider a random map $\chi:V(H)\to V(J)$ and to keep the edges $e\in E(H)$  which have $\chi(e)\in E(J)$.  However, one quickly sees that this $H'$ may not be $F$-free.  Indeed, if $F$ is $r$-partite, then it is possible for $\chi$ to map every edge of $H$ to a single edge of $J$, giving $H'=H$.
	
	We get around this issue by doing two additional steps.  The first is to put constraints on $H'$ to disallow edges $e$ which have $\chi(e)=\chi(f)$ for some $f$ with $|e\cap f|=r-1$.  The second is to choose $J$ so that it avoids not only $F$, but also a family of $r$-graphs related to $F$.  More precisely, for $r$-graphs $F,F'$ we say that a map $\phi:V(F)\to V(F')$ is a \textit{local isomorphism} if
	\begin{itemize}
		\item[$(a)$] $\phi$ is a homomorphism and
		\item[$(b)$] $\phi(e)\ne \phi (f)$ for $e,f\in E(F)$ with $|e\cap f|=r-1$.
	\end{itemize}    
	For example, with $F=K_{s_1,\ldots,s_r}$, a local isomorphism $\phi:V(F)\to V(F)')$ exists only when $F'$ contains a subgraph isomorphic to $F$, as mapping any two vertices of $F$ to the same vertex in $F'$ would cause two edges intersecting in $r-1$ spots to map to the same edge.  As another example, the figure below illustrates a local isomorphism from $C_8$ to two $C_4$'s sharing an edge (where the map sends the two star/diamond vertices on the left to the star/diamond vertex on the right),
	
	
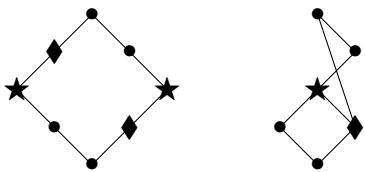
\[
	\begin{tikzpicture}[scale=.5]
	\node at (0,0) {$\bullet$};
	\node at (1,1) { $\blacklozenge$};
	\node at (-1,1) {$\bullet$};
	\node at (2,2)  { $\bigstar$};
	\node at (-2,2)  {$\bigstar$};
	\node at (1,3) {$\bullet$};
	\node at (-1,3) {$\blacklozenge$};
	\node at (0,4) {$\bullet$};
	
	\draw (0,0) -- (-1,1);
	\draw (0,0) -- (1,1);
	\draw (1,1) -- (2,2);
	\draw (2,2) -- (1,3);
	\draw (1,3) -- (0,4);
	\draw (-1,1) -- (-2,2);
	\draw (-2,2) -- (-1,3);
	\draw (-1,3) -- (0,4);
	
	\node at (6,0) {$\bullet$};
	\node at (7,1) {$\blacklozenge$};
	\node at (5,1) {$\bullet$};
	\node at (6,2) {$\bigstar$};
	\node at (7,3) {$\bullet$};
	\node at (6,4) {$\bullet$};
	
	\draw (6,4) -- (7,1);
	\draw (6,0) --(7,1);
	\draw (7,1) -- (6,2);
	\draw (6,2) -- (7,3);
	\draw (7,3) -- (6,4);
	\draw (6,0) -- (5,1);
	\draw (6,2) -- (5,1);
	
	\end{tikzpicture}
	\]
	For an $r$-graph $F$ we define the set $\HI{F}$ to be the set of $F'$ for which there exists a local isomorphism $\phi:V(F)\to V(F')$ such that both $\phi$ and its induced map on the edge sets are surjective.  For $\c{F}$ a family of $r$-graphs we define $\HI{\c{F}}=\bigcup_{F\in \c{F}}\HI{F}$.  For example, our argument above and the definition of $\c{H}$ implies that $\c{H}(K_{s_1,\ldots,s_r})=\{K_{s_1,\ldots,s_r}\}$.  Also observe that $|V(F')|\le |V(F)|$ for any $F'\in \c{H}(F)$.
	
	In general it turns out that we want to choose a $J$ which is not only $F$-free but also $\c{H}(F)$-free, and this gives our main technical result for hosts with low codegrees.
	
	\begin{lem}\label{lem:homGen}
		Let $\c{F}$ be a family of $r$-graphs and $H$ an $r$-graph with maximum $(r-1)$-degree at most $D$.  Then as $D\to \infty$,
		\[\ex(H,\c{F})=\Om\l(\f{\ex(D,\HI{\c{F}})}{D^{r}}\r)\cdot e(H).\]
	\end{lem}
	\begin{proof}
		Let $J$ be an extremal $\c{H}(\c{F})$-free $r$-graph on $t:=2r^2D$ vertices and let $\chi:V(H)\to V(J)$ be chosen uniformly at random.  Let $H'\sub H$ be the (random) subgraph which keeps the edge $e\in E(H)$ if
		\begin{itemize}
			\item[$(1)$] $\chi(e)\in E(J)$, and
			\item[$(2)$] $\chi(f)\ne \chi(e)$ for any other $f\in E(H)$ with $|e\cap f|=r-1$.
		\end{itemize}
		
		We claim that $H'$ is $\c{F}$-free. Indeed, assume $H'$ contained some subgraph $F'$ isomorphic to $F\in \c{F}$.  Let $F''$ be the subgraph of $J$ with $V(F'')=\{\chi(v):v\in V(F')\}$ and $E(F'')=\{\chi(e):e\in E(F')\}$, and we note that $F'\sub H'$ implies that each edge of $F'$ satisfies (1), so every element of $E(F'')$ is an edge in $J$.  By conditions (1) and (2), $\chi$ is a surjective local isomorphism from $F'$ to $F''$, so $F''\in \HI{F}$, a contradiction to our condition on $J$.
		
		It remains to compute how many edges $H'$ has in expectation.  Given an edge $e\in E(H)$, let $A$ be the event that $\chi(e)\in E(J)$, let $\hat{E}$ be the set of edges $f\in E(H)$ with $|e\cap f|=r-1$, and let $B$ be the event that $\chi(f)\not\sub \chi(e)$ for all $f\in \hat{E}$, which in particular implies (2) for the edge $e$.  It is easy to see that \[\Pr[A]=r!e(J)t^{-r}.\]  It is also not hard to see for any $f\in \hat{E}$ that $\Pr[\chi(f)\sub \chi(e)|A]=\f{r}{t}$.  Because $\ol{B}$ (the complement of $B$) is the union of the events $\chi(f)\sub \chi(e)$ for $f\in \hat{E}$, the union bound gives  \[\Pr[B|A]\ge 1-|\hat{E}|\cdot rt^{-1}\ge 1-r^2D t^{-1}= \half ,\]
		where the last inequality used that there are at most $D$ edges $f$ intersecting a given $r-1$ sized subset of $e$.  Thus
		\[\Pr[e\in E(H')]\ge \Pr[A]\cdot \Pr[B|A]\ge \half r!e(J)t^{-r},\]
		and by linearity of expectation we have $\E[e(H')]\ge \half r! e(J)t^{-r}\cdot e(H)$. We conclude that there exists some subgraph $H''\sub H$ with at least this many edges which is $\c{F}$-free, and the result follows since $e(J)=\ex(t,\HI{\c{F}})$ and $t=2r^2D$.
	\end{proof}

	Lastly, we observe the following.
	\begin{lem}\label{lem:label}
		If $F$ is a tightly connected $r$-graph, then it has an $r$-partition which is unique up to relabeling the parts.
	\end{lem}
	\begin{proof}
		By definition $F$ has an $r$-partition $U_1\cup \cdots \cup U_r$ such that for any distinct $u_1,u_2\in U_i$ there exist edges $e_1,e_2$ with $u_j\in e_j$ and $|e_1\cap e_2|=r-1$.  Let $e=\{x_1,\ldots,x_r\}\in E(F)$ with $x_i\in U_i$ for all $i$ and let $U'_1\cup \cdots \cup U'_r$ be any other $r$-partition of $F$ with $x_i\in U'_i$ for all $i$.  Note that if $y_i\in U_i\sm \{x_i\}$, then there is an edge $e'\ni y_i$ with $|e\cap e'|=r-1$.  Thus for $e'$ to contain exactly one vertex from each $U'_j$ set we must have $y_i\in U'_i$.  This implies that $U_i\sub U'_i$ for all $i$, and hence $U'_i=U_i$ for all $i$.  We conclude that every $r$-partition of $F$ is a relabeling of the partition $U_1\cup \cdots \cup U_r$.
	\end{proof}
	
	\textbf{Proof of Theorem~\ref{thm:homGen}}.  Let $F$ be tightly connected  with $r$-partition $U_1\cup \cdots \cup U_r$.  We claim that if $\chi:V(F)\to V(F')$ is a local isomorphism, then either $F'\cong F$ or $F'$ is not $r$-partite.  If $\chi$ is injective then $F'\cong F$, so assume $\chi(u_1)=\chi(u_2)$ for some $u_1\ne u_2$ and that $F'$ has some $r$-partition $U'_1\cup \cdots \cup U'_r$.  Observe that $\chi$ being a homomorphism implies $\chi(U'_1)^{-1}\cup \cdots \cup \chi(U'_r)^{-1}$ is an $r$-partition of $F$, and by Lemma~\ref{lem:label} we can relabel parts so that $\chi(U'_i)^{-1}=U_i$.  Assume $\chi(u_1)=\chi(u_2)\in U'_i$, meaning $u_1,u_2\in U_i$.  Then $F$ being tightly connected means there exist edges $e_1,e_2\in E(F)$ with $|e_1\cap e_2|=r-1$ and $u_i\in e_i$. Thus $\chi(e_1)=\chi(e_2)$, contradicting $\chi$ being a local isomorphism.
	
	We conclude that for tightly connected $F$ that $\HI{F}$ consists of $F$ together with some $r$-graphs which are not $r$-partite, so if $\c{F}$ is a family of tightly connected $r$-graphs then $\HI{\c{F}}$ consists of $\c{F}$ together with some $r$-graphs which are not $r$-partite.  Thus given any extremal $\c{F}$-free $r$-graph $H$ on $n$ vertices, we can consider a maximum $r$-partite subgraph of $H$ which has at least $r^{-r}e(H)$ edges and which is $\HI{\c{F}}$-free.  This implies $\ex(n,\HI{\c{F}})\ge r^{-r}\ex(n,\c{F})$ for all such $\c{F}$, and the result follows from Lemma~\ref{lem:homGen}. \qed
	
	\section{Hosts with High Codegrees : Proofs of Theorems \ref{thm:codeg} and \ref{thm:tightTuran}}\label{sec:codeg}
	
	For the proofs of Theorems  \ref{thm:codeg} and \ref{thm:tightTuran}, the central idea will be to find an $r$-partite subgraph $H'\sub H$ on $V_1\cup \cdots \cup V_r$ such that the $k$-graph induced by $V_1\cup \cdots \cup V_k$ avoids certain $k$-graphs associated to $\c{F}$.  For Theorem~\ref{thm:tightTuran} we use $k=2$, and for Theorem~\ref{thm:codeg} we use $k=r-1$.

	{\bf Proof of  Theorem~\ref{thm:tightTuran}.} Recall $TC_{k}^r$ is defined on $\{u_0,\ldots,u_{k - 1}\}$ with edges $\{u_i,u_{i+1},\ldots,u_{i+r-1}\}$. In order to later prove Theorem~\ref{thm:tightGen}, it will be convenient to prove the slightly more technical result
	\begin{equation}\label{eq:tight}
	\ex(n,\HI{TC_{r+1}^{r},\ldots,TC_{\ell r}^{r}})=\Om(n^{r-1+1/\ell})\tr{ for }r\ge 2,\ \ell\in \{2,3,5\},
	\end{equation}
	which implies Theorem~\ref{thm:tightTuran} since $\c{F}\sub \HI{\c{F}}$ for all $\c{F}$.
	
	As noted in the introduction, for $\ell\in \{2,3,5\}$ it is known \cite{Benson} that there exists a bipartite graph $G$ on $V_1\cup V_2$ with $|V_i|=n/r$ which is $\{C_3,\ldots,C_{2\ell}\}$-free and such that $e(G)=\Om(n^{1+1/\ell})$.  Define the $r$-graph $H$ on $V_1\cup \cdots \cup V_r$ with $|V_i|=n/r$ by including all edges $e$ with $e\cap (V_1\cup V_2)\in E(G)$.  Then $e(H)=(n/r)^{r-2}\cdot e(G)=\Om(n^{r-1+1/\ell})$, so it suffices to prove that $H$ is $\HI{TC_{r+1},\ldots,TC_{\ell r}^r}$-free.
	
	Assume there exists $F\sub H$  such that there exists a local isomorphism $\chi:V(TC_{k}^r)\to V(F)$ for some $k\le \ell r$.   Since $H$ is $r$-partite, $F$ has an $r$-partition $U_1\cup \cdots \cup U_r$ given by $U_i=V_i\cap V(F)$.  Because $\chi$ is a homomorphism, $\chi^{-1}(U_1)\cup \cdots \cup \chi^{-1}(U_r)$ is an $r$-partition of $TC_{k}^r$, and in particular we must have $k=\ell'r$ for some $\ell'\le \ell$.  Further, because $TC_{\ell' r}^r$ has a unique $r$-partition up to relabeling its parts for all $\ell'$, we can assume without loss of generality that $W_i:=\chi^{-1}(U_i)$ consists of all the $u_j$ vertices with $j\equiv i \mod r$.
	
	Observe that the graph induced by $W_1\cup W_2$ in $TC_{\ell' r}^r$ is a $C_{2\ell'}$ on $u_1u_2u_{r+1}u_{r+2}\cdots u_{(\ell' -1)r+2}$.  Further note that the restricted map $\chi:W_1\cup W_2\to U_1\cup U_2$ is a local isomorphism of $C_{2\ell'}$, as any two edges of this $C_{2\ell '}$ which intersect in 1 vertex are contained in two edges of $TC_{\ell' r}^r$ which intersect in $r-1$ vertices.  Thus the graph induced by $U_1\cup U_2=V(F)\cap (V_1\cup V_2)$ in $F$ is a local isomorphism of $C_{2\ell'}$, and in particular this graph must contain some cycle of length $C_{k'}$ with $k'\le 2\ell'$ as a subgraph.  But the graph induced by $V(F)\cap (V_1\cup V_2)$ is a subgraph of $G$ (since $G$ is the graph induced by $V_1\cup V_2$), which is $C_{k'}$-free by construction, a contradiction.  We conclude that $H$ contains no element of $\HI{TC_{k}^r}$ for $k\le \ell r$, proving \eqref{eq:tight}. \qed
	
	\medskip

	{\bf Proof of Theorem \ref{thm:codeg}.} Assume $H$ has $r$-partition $V_1\cup \cdots \cup V_r$ and let $E_{i}\sub E(H)$ denote the set of edges $e=\{v_1,\ldots,v_r\}$ with $v_j\in V_j$ for all $j$ such that $e\sm \{v_i\}$ has $(r-1)$-degree at least $D$.  By hypothesis we have \[\half e(H)\le \l|\bigcup E_{i}\r|\le \sum |E_{i}|,\] so we can assume without loss of generality that $|E_r|\ge (2r)^{-1}e(H)$.  Let $G_k$ denote the $(r-1)$-graph on $V_1\cup \cdots\cup V_{r-1}$ with \[E(G_k)=\{\{v_1,\ldots,v_{r-1}\}:2^k D\le d_H(v_1,\ldots,v_{r-1})<2^{k+1}D\}.\] That is, the $G_k$ roughly partition the edges of $E_r$ into subgraphs which are codegree regular.  We say that the edge $e\in E_{r}$ {\em corresponds to} the edge $e'\in E(G_k)$ if $e'\sub e$.
	
	Note that the $(r-1)$-degree of any set of $r-1$ vertices in $H$ is at most $\Del$, so we only have to consider $G_k$ with  $k=O(\log \Del)$.  As each edge in $E_{r}$ corresponds to an edge in exactly one $G_k$, by the pigeonhole principle there exists some
	$K$ in this range such that at least $\Om((\log \Del)^{-1})e(H)$ edges of $E_{r}$ correspond to edges in $G_K$.  Because each edge of $G_K$ is corresponded to by at most $2^{K+1}D$ edges of $E_{r}$, we have $e(G_K)=\Om((2^{K} D\log\Del)^{-1})e(H)$.  Because each edge of $G_K$ is corresponded to by at least $2^K D$ edges of $E_{r}$, the maximum degree of $G_K$ is at most $O((2^{K}D)^{-1}\Del)$.  Let $\tilde{\Del}$ be three times the maximum degree of $G_K$ (so that $\tilde{\Del}\ge 3$) and let $G\sub G_K$ be an $\boldsymbol{\pi}(\c{F})$-free subgraph of $G_K$ with as many edges as possible.  By the hypothesis of the theorem, we have
	\begin{align*}e(G)&=\ex(G_K,\boldsymbol{\pi}(\c{F}))=\Om( \tilde{\Del}^{-\gam}\cdot (\log \tilde{\Del})^{3-r}) \cdot e(G_K)\\&=\Om( 2^{\gam K}(\Del/D)^{-\gam}\cdot (\log \Del)^{3-r} \cdot (2^{K} D\log\Del)^{-1})\cdot e(H)\\&= 2^{-K}D^{-1}\cdot \Om(2^{\gam K}(\Del/D)^{-\gam}(\log \Del)^{2-r})\cdot  e(H).
	\end{align*}
	
	Define $H'\sub H$ to have the edges which correspond to edges of $G$.  As each edge in $G\sub G_K$ is corresponded to by at least $2^KD$ edges, $e(H')\ge 2^K D e(G)$.  This is sufficiently large to prove the result, so it will be enough to show that $H'$ is  $\c{F}$-free.  Indeed, assume $H'$ contained some $F'\cong F\in \c{F}$ with $r$-partition $U_1\cup \cdots \cup U_r$ given by $U_i=V(F')\cap V_i$.  In particular $G$, the $(r-1)$-graph of $H'$ induced by parts $V_1\cup \cdots \cup V_{r-1}$, contains the $(r-1)$-graph of $F'$ induced by $U_1\cup \cdots \cup U_{r-1}$ as a subgraph, a contradiction to $G$ being $\boldsymbol{\pi}(F)\sub \boldsymbol{\pi}(\c{F})$-free. \qed
	
	We note that one can replace the $\log \Del$ terms in Theorem~\ref{thm:codeg} with $\log\log \Del$ using a slightly more refined argument, namely by partitioning the edge set by $2^{(1+\gam)^k} D\le d_H(v_1,\ldots,v_{r-1})<2^{(1+\gam)^{k+1}}D$ with $\gam$ as in the theorem statement.
	
	\section{Proofs of Theorems~\ref{thm:K2} and~\ref{thm:tightGen}}\label{sec:proofs}
	
	\subsection{Proof of Theorem \ref{thm:tightGen}}
	
	Let $\c{F}_{\ell,r}=\{TC_{r+1}^r,\ldots,TC_{\ell r}^r\}$.  For $\ell\in \{2,3,5\}$, we prove by induction on $r$ the lower bound
	\[ \ex(H,\c{F}_{\ell,r})=\Om(\Del^{\f{-\ell+1}{\ell(r-1)}}(\log \Del)^{2-r})e(H).\] The case $r = 2$ was established in~\cite{FKP} (and it can also be proven by Lemma~\ref{lem:homGen}), so we assume the result holds for $(r - 1)$-graphs. Let $H$ be an $r$-graph of maximum degree $\Del$ and let $H'\sub H$ be an $r$-partite subgraph with at least $r^{-r} e(H)$ edges.  Let $D:=\Del^{1/(r - 1)}$, and let $H''\sub H'$ contain all the edges which do not contain an $(r-1)$-set with $(r-1)$-degree at least $D$.  If $e(H'')\ge \half e(H')$, then because $H''$ has maximum $(r-1)$-degree at most $D$, we can apply Lemma~\ref{lem:homGen} to $H''$ to conclude by \eqref{eq:tight} that
	\begin{equation*}
	\ex(H,\c{F}_{\ell,r})= \Om( D^{-1+ \frac{1}{\ell}})\cdot \half r^{-r}e(H)= \Om(\Del^{\f{-\ell+1}{\ell(r-1)}}) \cdot e(H),
	\end{equation*}
	giving the desired bound.
	
	Thus we can assume that at least half the edges of $H'$ contain an $(r-1)$-set with $(r-1)$-degree at least $D$.  Let $\c{F}_{\ell,r}'\sub \c{F}_{\ell,r}$ by all the tight cycles of the form $TC_{\ell'r}^r$, and note that $\c{F}_{\ell,r}\sm \c{F}_{\ell,r}'$ consists of $r$-graphs which are not $r$-partite, so $H'$ automatically avoids these.  Because $TC_{\ell' r}^r$ has a unique $r$-partition up to relabeling its parts, it is straightforward to check $\boldsymbol{\pi}(TC_{\ell' r}^r)=\{TC_{\ell'(r - 1)}^{r-1}\}$, and thus $\boldsymbol{\pi}(\c{F}'_{\ell,r})=\c{F}'_{\ell,r-1}\sub \c{F}_{\ell,r-1}$.  By  Theorem~\ref{thm:codeg} and the inductive hypothesis we conclude \begin{align*}
	\ex(H',\c{F}_{\ell,r})&=\ex(H',\c{F}_{\ell,r}')=\Om(\Del^{\f{-\ell+1}{\ell(r-2)}}D^{\f{\ell-1}{\ell(r-2)}}(\log \Del)^{2-r})\cdot r^{-r}e(H)\\&=\Om( \Del^{\f{-\ell+1}{\ell(r-1)}}(\log \Del)^{2-r}) \cdot e(H).
	\end{align*}
	\qed

	\subsection{Proof of Theorem \ref{thm:K2}: Lower Bound}\label{sec:low}
	
	Let $2\le s_1\le \cdots \le s_r$ be such that $\ex(n,K_{s_1,\ldots,s_i})=\Om(n^{i-1/a_i})$ for all $2\le i\le r$, and it is known that this holds provided $s_{i}\ge f_i(s_{i-1})$ for some suitable function $f_i$ \cite{MYZ,PZ}.  We prove $\ex(H,K_{s_1,\ldots,s_r})=\Om(\Del^{-1/\sum_{i=2}^r a_i}(\log \Del)^{2-r})$ by induction on $r$.  The case $r=2$ comes from Theorem~\ref{thm:homGen} and the assumption $\ex(n,K_{s_1,s_2})=\Om(n^{2-1/a_2})$. We proceed to $r > 2$.
	
	Let $H'\sub H$ be an $r$-partite subgraph with at least $r^{-r} e(H)$ edges.  Let $D:=\Del^{a_r/\sum_{i=2}^r a_i}$, and let $H''\sub H'$ contain all the edges which do not contain an $(r-1)$-set with $(r-1)$-degree at least $D$.  If $e(H'')\ge \half e(H')$, then because $H''$ has maximum $(r-1)$-degree at most $D$, we can apply Theorem~\ref{thm:homGen} and the hypothesis $\ex(n,K_{s_1,\ldots,s_r})=\Om(n^{r-1/a_r})$ to $H''$ and conclude \begin{equation*}\ex(H,K_{s_1,\ldots,s_r})=\Om( D^{-1/a_r})e(H)=\Om\l( \Del^{\f{-1}{\sum_{i=2}^r a_i}}\r)e(H),\end{equation*}
	giving the desired bound.
	
	Thus we can assume that at least half the edges of $H'$ contain an $(r-1)$-set with $(r-1)$-degree at least $D$.  Because $\boldsymbol{\pi}(K_{s_1,\ldots,s_r})$ is a set of $(r-1)$-graphs which all contain $K_{s_1,\ldots,s_{r-1}}$ as a subgraph, we can use Theorem~\ref{thm:codeg} and the inductive hypothesis to conclude \begin{align*}
	\ex(H',K_{s_1,\ldots,s_r})&=\Om\l( \Bigl(\frac{\Del}{D}\Bigr)^{\f{-1}{\sum_{i=2}^{r-1} a_i}}(\log \Del)^{2-r}\r)e(H)\\&= \Om\l(\Del^\f{-1}{\sum_{i=2}^r a_i}(\log \Del)^{2-r}\r)e(H),
	\end{align*}
	where the equality $(\Del/D)^{-1/\sum_{i=2}^{r-1} a_i}=\Del^{-1/\sum_{i=2}^{r} a_i}$ used
	\[\log_{\Del}\Bigl(\frac{\Del}{D}\Bigr)=1-\f{a_r}{\sum_{i=2}^ra_i}=\f{\sum_{i=2}^{r-1} a_i}{\sum_{i=2}^ra_i}.\]
	This proves the desired lower bound. \qed
	
	\subsection{Supersaturation for complete $r$-partite $r$-graphs}
	
	A closer inspection of the proof of Theorem~\ref{thm:K2}'s lower bound shows that if this bound were to be sharp, then the corresponding construction must have essentially all of its edges containing a $k$-set with $k$-degree $\Del^{a_{k+1}/(\sum_{i=2}^{k+1}a_i)}$.  To prove Theorem \ref{thm:K2}'s upper bound,
	we use the following supersaturation result for a certain unbalanced $r$-graph which has $k$-degrees of this form.
	
	\begin{lem}\label{lem:construct}
		Let $2\le s_1\le \cdots \le s_r$ and define $a_i=\prod_{j<i}s_j$ for $2\le i\le r+1$.  Let $H$ be a complete $r$-partite $r$-graph on $V_1\cup \cdots \cup V_r$ with $n_i:=|V_i|$ defined by $n_1=n^{a_2}$ and $n_i=n^{a_i}$ otherwise.  There exists a constant $\al_r\ge 1$ such that if $H'\sub H$ has $m\ge \al_r n^{-1}\prod_{i= 1}^r n_i $ edges, then $H'$ contains at least $\al_r^{-1} m^{a_{r+1}} \prod_{i=1}^r n_i^{s_i-a_{r+1}}$ copies of $K_{s_1,\ldots,s_r}$.
	\end{lem}
	Strictly speaking, to prove Theorem~\ref{thm:K2} we only need that there exists at least one copy of $K_{s_1,\ldots,s_r}$ when $H$ contains at least this many edges, but it is easier to prove Lemma~\ref{lem:construct} by induction on $r$ if we add this stronger conclusion.
	\begin{proof}
		We prove this by induction on $r$.  For $r=2$, we use the following supersaturation result of Erd\H{o}s and Simonovits \cite{ES}: for $s_1\le s_2$, there exist constants $\al,\al'$ such that if $G\sub K_{N,N}$ has  $e(G)=m\ge \al' N^{2-1/s_1}$, then $G$ contains at least $\al^{-1} m^{s_1s_2}N^{s_1+s_2-2s_1s_2}$ copies of $K_{s_1,s_2}$.  Substituting $N=n_1=n_2=n^{a_2}$ and taking $\al_2=\max\{\al',\al,1\}$ gives the result.
		
		Assume the result holds up to but not including $r$.  Let $\c{Q}_i$ denote the set of all subsets of $V_i$ of size $s_i$ and let $\c{Q}=\c{Q}_1\times \cdots \times\c{Q}_{r-1}$. Let $\c{P}$ denote the set of pairs $(Q,v)$ with $Q\in \c{Q}$ and $v\in V_r$ such that $\{v_1,\ldots,v_{r-1},v\}\in E(H)$ for all $v_i\in Q_i$.  That is, the vertices of $Q$ induce a $K_{s_1,\ldots,s_{r-1}}$ in the link graph of $v$.  For each $Q\in \c{Q}$, define $g_Q$ to be the number of pairs in $\c{P}$ using the set $Q$, and similarly define $f_v$ for $v\in V_r$ to be the number of pairs in $\c{P}$ using $v$.  Observe that the number of $K_{s_1,\ldots,s_r}$'s in $H'$ is at least $\sum_{Q\in\c{Q}} {g_Q\choose s_r}$, so it will be enough to lower bound this sum.
		
		For non-negative $x$ define ${x\choose s_r}=0$ if $x\le s_r-1$ and ${x\choose s_r}=\f{x(x-1)\cdots(x-k+1)}{s_r!}$ otherwise.  This makes ${x\choose s_r}$ a convex function, so we have {\small \begin{align}\sum_{Q\in \c{Q}}{g_Q\choose s_r}&\ge |\c{Q}|{|\c{Q}|^{-1}\sum_{Q\in \c{Q}} g_Q\choose s_r}=|\c{Q}|{|\c{Q}|^{-1}\sum_{v\in V_r} f_v\choose s_r}\nonumber\\&\ge \f{|\c{Q}|}{s_r!}\l(|\c{Q}|^{-1}\sum_{v\in V_r} f_v-s_r\r)^{s_r}.\label{eq:P1}\end{align}  }
		
		Assume $m\ge \al_r n^{-1}\prod_{i\le r} n_i$ where  \begin{equation}\label{eq:al} \al_r:=4^{s_r}\al_{r-1}^{s_r}s_r!\ge \max\{ 2^{1/a_r}\al_{r-1},\ (4s_r \al_{r-1})^{1/a_{r}}\}.\end{equation}
		For non-negative $x$ define \[h(x)=\al_{r-1}^{-1} x^{a_r}\prod_{i<r} n_i^{s_i-a_r}.\]  Let $d_v$ denote the degree of $v$ in $H'$, and define $m'=\al_{r-1}n^{-1}\prod_{i<r}n_i$.  By the inductive hypothesis, the number of copies of $K_{s_1,\ldots,s_{r-1}}$ in the link graph of $v\in V_r$ will be at least $h(d_v)$ whenever $d_v\ge m'$, and since $h(x)$ is an increasing function we have \begin{equation*}f_v\ge h(d_v)-h(m').\end{equation*}  From \eqref{eq:al} we have \[m/n_r\ge 2^{1/a_r} m'.\]
		Using these two observations and the convexity of $h$, we find
		\begin{equation}\label{eq:P2}\sum_{v\in V_r} f_v\ge \sum_{v\in V_r}(h(d_v)-h(m'))\ge n_r\cdot  (h(m/n_r)-h(m'))\ge \half n_r\cdot h(m/n_r).\end{equation}
		
		Since $|\c{Q}|\le \prod_{i<r} n_i^{s_i}$ and $n_r=n^{a_r}$ for $r\ge 3$, we have by \eqref{eq:al} \[\half n_r\cdot h(m/n_r)|\c{Q}|^{-1}\ge \half n^{a_r}\cdot h\l(\al_r n^{-1}\prod_{i<r}n_i\r)\prod_{i<r} n_i^{-s_i}=\half \al_{r-1}^{-1} \al_r^{a_{r}}\ge 2s_r.\]
		
		Combining this with \eqref{eq:P2} gives  \begin{align}|\c{Q}|^{-1}\sum_{v\in V_r} f_v-s_r&\ge\half n_r\cdot h(m/n_r)|\c{Q}|^{-1}-s_r\nonumber\\&\ge \quart n_r \cdot h(m/n_r)\cdot |\c{Q}|^{-1+1/s_r}|\c{Q}|^{-1/s_r}\nonumber
		\\&\ge \quart n_r \cdot \al_{r-1}^{-1} m^{a_r}n_r^{-a_r}\prod_{i<r} n_i^{s_i-a_r}\cdot \prod_{i<r} n_i^{-s_i+s_i/s_r}|\c{Q}|^{-1/s_r}\nonumber
		\\&= (4\al_{r-1})^{-1} m^{a_r} \prod_{i=1}^rn_i^{s_i/s_r-a_r} |\c{Q}|^{-1/s_r},\label{eq:P3}\end{align}
		where this last inequality again used $|\c{Q}|\le \prod n_i^{s_i}$ and that $-1+1/s_r\le 0$ since $s_r\ge 1$.
		
		Plugging \eqref{eq:P3} into \eqref{eq:P1} gives
		\begin{align*}
		\sum_{Q\in\c{Q}}{g_Q\choose s_r}&\ge \f{|\c{Q}|}{s_r!}\cdot (4\al_{r-1})^{-s_r} m^{a_{r+1}} \prod_{i=1}^r n_i^{s_i-a_{r+1}}|\c{Q}|^{-1}=\al_r^{-1} m^{a_{r+1}} \prod_{i=1}^r n_i^{s_i-a_{r+1}}.
		\end{align*}
		Thus $H'$ contains at least this many copies of $K_{s_1,\ldots,s_r}$ as desired.
	\end{proof}

	\subsection{Proof of Theorem \ref{thm:K2}: Upper Bound}\label{sec:constructions}
	
	We define the $r$-graph $\Con{a_2,\ldots,a_r}$ as follows.  $\Con{a_2}$ is the complete balanced bipartite graph with parts of size $n^{a_2}$.  If $\Con{a_2,\ldots,a_{r-1}}$ has been defined, then we construct $\Con{a_2,\ldots,a_r}$ as follows. Take a set $\{H_1,H_2,\ldots,\}$ containing $n^{a_r-a_{r-1}}$ disjoint copies of $\Con{a_2,\ldots,a_{r-1}}$ and let $V_r$ consist of $n^{a_r}$ new vertices.  Define $\Con{a_2,\ldots,a_r}$ to have vertex set $V$ and edge set $E$, where \[V=V_r\cup \bigcup V(H_i),\]\[ E=\{ e\cup\{v\}:e\in  \bigcup E(H_i), v\in V_r\}.\]
	
	Let us examine some basic properties of $H_r:=\Con{a_2,\ldots,a_r}$.  By construction we see that $e(H_r)=n^{a_r}\cdot n^{a_r-a_{r-1}} \cdot e(\Con{a_2,\ldots,a_{r-1}})$, and thus inductively one can prove that $e(H_r)=n^{2a_r+a_{r-1}+\cdots+a_2}$ for $r\ge 3$.  Similarly for $r\ge 2$ each vertex in $V_r$ has degree $n^{a_r+\cdots+a_2}:=\Del$, and it is easily seen that $H_r$ is $\Del$-regular.  To prove the upper bound of Theorem \ref{thm:K2}, it suffices to show \[\ex(H_r,K_{s_1,\ldots,s_r})=O(\Del^{-1/\sum_{i=2}^ra_i})e(H_r)=O(n^{2a_r+a_{r-1}+\cdots+a_2-1}).\]
	
	Let $\al_r$ be as in Lemma~\ref{lem:construct}, and let $H'_r\sub H_r$ be a subgraph with $e(H'_r)\ge \al_r n^{2a_r+a_{r-1}+\cdots+a_2-1}$ which contains no $K_{s_1,s_2,\ldots,s_r}$.  By the pigeonhole principle, one of the copies of $\Con{a_2,\ldots,a_{r-1}}$ making up $H_r$ is involved with at least $ n^{a_{r-1}-a_r}e(H'_r)$ edges of $H'_r$.  Let $H_{r-1}$ denote such a copy, and let $H'_{r-1}$ be $H'_r$ after deleting every copy of $\Con{a_2,\ldots,a_{r-1}}$ that is not $H_{r-1}$.  Now again in $H'_{r-1}$ there exists some copy of $\Con{a_2,\ldots,a_{r-2}}$ involved with at least $n^{a_{r-2}-a_{r-1}}e(H'_{r-1})$ edges of $H'_{r-1}$. Call this copy $H_{r-2}$, and let $H'_{r-2}$ be $H'_{r-1}$ after deleting every copy that is not $H_{r-2}$.  Continue this way until one arrives at $H'_{2}$.
	
	Observe that $H'_2$ is $r$-partite with the $i$th part having size $n_i:=n^{a_i}$ for $i\ge 2$ and $n_1=n^{a_2}$.  Because $e(H'_{i-1})\ge e(H'_{i})n^{a_{i-1}-a_{i}}$ for all $i$, we have \[e(H'_{2})\ge e(H'_r) n^{a_2-a_r}\ge \al_r n^{a_r+\cdots+a_3+2a_2-1}=\al_r n^{-1}\prod_{i\le r} n_i.\]  Then
	$H'_2$ contains a $K_{s_1,\ldots,s_r}$ by Lemma~\ref{lem:construct}, contradicting the assumption that $H_r'$ is $K_{s_1,\ldots,s_r}$-free.  Thus there exists no such subgraph of $H_r$ with this many edges, giving the desired result. \qed

	\section{Random Hosts}\label{sec:rand}
	
	Throughout this section we fix integers $2\le s_1\le \cdots \le s_r$. Recall that $a_r=\prod_{i<r}s_i,\ a_{r+1}=\prod_{i\le r} s_i$, and
	\[\be_1=\f{\sum_{i=1}^r s_i-r}{a_{r+1}-1} \quad \quad \mbox{ and } \quad \quad \ \be_2=\f{a_r(\sum_{i=1}^{r-1}s_i-r)+1}{(a_r-1)(a_{r+1}-1)}.\]
	
	\subsection{Proof of Theorem~\ref{thm:randK} : Lower Bounds}

	An ingredient of the proof of Theorem \ref{thm:randK} is the following version of Azuma's inequality, a proof of which can be found in \cite{ProbMeth}, for example.
	
	\begin{lem}\label{lem:azuma}
		Let $X_1,X_2,\dots,X_N$ be independent Bernoulli random variables and let $f : \{0,1\}^N \rightarrow \mathbb R$ be a function
		such that $|f(x) - f(y)| \leq 1$ whenever $x$
		and $y$ differ only on the $i$th co-ordinate. Then with $Z =
		f(X_1,X_2,\dots,X_N)$, we have for all $\lam>0$
		\begin{equation*}
		\mathbb P(|Z - \mathbb E(Z)| > \lambda)  \leq  2\mbox{e}^{-2\lambda^2/N}.
		\end{equation*}
	\end{lem}
	For example, if $f(X_1,X_2,\dots,X_N) = X_1 + X_2 + \dots + X_N$ then we obtain a form of the Chernoff Bound for binomial random variables. In all of our applications the $X_i$ will be the indicator function for whether the $i$th edge of $K_n^r$ is in $H_{n,p}^r$.
	\medskip
	
	We first consider $n^{-r/2}\log n\le p\le cn^{-\be_1}$ for some small constant $c$.  Observe that $\E[e(H_{n,p})]=\Theta(p n^r)$, and for $c$ sufficiently small this is at least twice the expected number of copies of $K_{s_1,\ldots,s_r}$ in $H_{n,p}^r$ which is $\Theta(p^{a_{r+1}}n^{\sum_{i=1}^r s_i})$.  Thus by deleting an edge from each copy of $K_{s_1,\ldots,s_r}$ in $H_{n,p}^r$ we see that $\E[\ex(H_{n,p}^r,K_{s_1,\ldots,s_r})]=\Om(pn^r)$.  Observe that $Z=\ex(H_{n,p}^r,K_{s_1,\ldots,s_r})$ satisfies the conditions\footnote{To be somewhat more explicit, we let $X_i=1$ if the $i$th edge of $K_n^r$ is in $H_{n,p}^r$ and $X_i=0$ otherwise.  Then $Z$ is a function of the $X_i$, and changing one value of $X_i$ (i.e. adding or deleting an edge in $H_{n,p}^r$) changes the value of $Z$ (i.e. the size of a largest $K_{s_1,\ldots,s_r}$-free subgraph of $H_{n,p}^r$) by at most 1.} of Lemma~\ref{lem:azuma}, so taking $\lam= c' n^{r/2}\sqrt{\log n}$ for some small $c'$ (and using $\E[Z]\ge \Om(n^{r/2}\log n)$ since $p\ge n^{-r/2}\log n$) shows that the probability that $Z$ is within a constant factor of its expectation tends to 1, giving the a.a.s.\ result.  Moreover, the bound $\ex(H_{n,p}^r,K_{s_1,\ldots,s_r})=\Om(n^{r-\be_1})$ continues to hold a.a.s.\ for $p\ge cn^{-\be_1}$ by the monotonicity of the relative Tur\'an number as a function of $p$.
	
	Next we show  $\ex(H_{n,p}^r,K_{s_1,\ldots,s_r}) = \Omega(p^{1-1/a_r}n^{r-1/a_r})$ when $p\ge n^{-\be_2}$. We observe that the $(r-1)$-degree of any $(r-1)$-set is a binomial random variable with $n-r+1$ trials and probability $p$.  Because $\be_2<1$ (which can be proven using the inequality of arithmetic and geometric means), one can show using Lemma~\ref{lem:azuma} that a.a.s.\ $H_{n,p}^{r}$ has maximum $(r-1)$-degree at most $O(pn)$, and also that a.a.s.\ $e(H_{n,p}^r)=\Om(pn^r)$. Thus if $\ex(n,K_{s_1,\ldots,s_r})=\Om(n^{r-1/a_r})$, by Theorem~\ref{thm:homGen} we have a.a.s.\  \[\ex(H_{n,p}^r,K_{s_1,\ldots,s_r})= \Om((pn)^{-1/a_r}pn^r)=\Om(p^{1-1/a_r}n^{r-1/a_r}).\]
	This proves the lower bounds of Theorem \ref{thm:randK}. \qed
	
	\subsection{Proof of Theorem~\ref{thm:randK} : Upper Bounds}
	
	The upper bound of Theorem~\ref{thm:randK} for small $p$ follows since a.a.s.\ $H_{n,p}^{r}$ has at most $O(pn^r)$ edges, and the result for $p$ in the middle range will follow from the large range since $\E[\ex(H_{n,p}^{r},K_{s_1,\ldots,s_r})]$ is non-decreasing in $p$.  Thus we can assume $p\ge n^{-\be_2}(\log n)^{2a_r/(a_r-1)}$.
	
	Our approach for the upper bounds in this range borrows heavily from the argument used by Morris and Saxton~\cite{MS} for the case $r=2$. To this end, we let $\c{I}=\c{I}(n)$ denote the collection of $K_{s_1,\ldots,s_r}$-free $r$-graphs on $n$ vertices, and let $\c{G}=\c{G}(n,k)$ denote the collection of all $r$-graphs with $n$ vertices and at most $k n^{r-1/a_r}$ edges.  The following lemma is the main technical result we need to prove our upper bounds, where by a colored $r$-graph we mean an $r$-graph together with an arbitrary labeled partition of its edge set
	
	\begin{lem}\label{lem:randTech}
		For any $2\le s_1\le \cdots\le s_r$, there exists a constant $c>0$ such that the following holds for sufficiently large $n,k$ with $k\le n^{\be_2/a_r}(\log n)^{2/(a_r-1)}$.  There exists a collection $\c{S}$ of colored $r$-graphs with $n$ vertices and at most $ck^{1-a_r}n^{r-1/a_r}$ edges and functions $g:\c{I}\to \c{S}$ and $h:\c{S}\to \c{G}(n,k)$ with the following properties:
		
		\begin{itemize}
			\item[$(a)$] For every $s\ge 0$, the number of colored $r$-graphs in $\c{S}$ with $s$ edges is at most
			\[\l(\f{c n^{r-1/a_r}}{s}\r)^{a_r s/(a_r-1)}\exp(c k^{1-a_r}n^{r-1/a_r}).\]
			\item[$(b)$] For every $I\in \c{I}$, we have $g(I)\sub I\sub g(I)\cup h(g(I))$.
		\end{itemize}
	\end{lem}
	Before proving Lemma~\ref{lem:randTech}, we first illustrate how it implies the upper bound of Theorem~\ref{thm:randK} when $p\ge n^{-\be_2}(\log n)^{2a_r/(a_r-1)}$.

	Define $k=p^{-1/a_r}$, which means $k\le n^{\be_2/a_r}(\log n)^{-2/(a_r-1)}$, and thus we can define $\c{S},g,h$ as in Lemma~\ref{lem:randTech}.  If there exists a $K_{s_1,\ldots,s_r}$-free subgraph $I\sub H_{n,p}^{r}$ with $m$ edges, then in particular $g(I)\sub H_{n,p}^{r}$ and $H_{n,p}^{r}$ contains at least $m-e(g(I))$ edges of $h(g(I))$.  Thus for $m\ge 2 k^{1-a_r}n^{r-1/a_r}$, the probability of $I\sub H_{n,p}^r$ for some $K_{s_1,\ldots,s_r}$-free $I$ with $m$ edges is at most
	{\small \begin{align*}
		\sum_{S\in \c{S}} p^m{kn^{r-1/a_r}\choose m-e(S)}&\le \sum_{s=0}^{ck^{1-a_r}n^{r-1/a_r}}\sum_{S\in \c{S},\ e(S)=s}p^s\cdot \l(\f{e pkn^{r-1/a_r}}{m-s}\r)^{m-s} \\
		&\le \sum_{s=0}^{c k^{1-a_r}n^{r-1/a_r}}\l(\f{c p^{(a_r-1)/a_r}n^{r-1/a_r}}{s}\r)^{a_r s/(a_r-1)} \exp(c k^{1-a_r}n^{r-1/a_r})\cdot \l(\f{e pkn^{r-1/a_r}}{m-s}\r)^{m-s}\\
		&\le \exp\l[O(1)\cdot (p^{(a_r-1)/a_r}n^{r-1/a_r}+k^{1-a_r}n^{r-1/a_r})\r]\cdot \l(\f{2e pkn^{r-1/a_r}}{m}\r)^{m/2}\\  &=\exp\l[O(1)\cdot p^{1-1/a_r}n^{r-1/a_r}\r]\cdot \l(\f{2e p^{1-1/a_r}n^{r-1/a_r}}{m}\r)^{m/2},
		\end{align*}}
	where the second inequality used Lemma~\ref{lem:randTech}(a), to get the last inequality we used that $(d/s)^{s}\le e^{d/e}$ and that $m-s\ge \half m$ for $m\ge 2k^{1-a_r}n^{r-1/a_r}$, and the last equality used $k=p^{-1/a_r}$.  This quantity will tend to 0 as $n$ towards infinity provided $m\ge c'p^{1-1/a_r}n^{r-1/a_r}$ for some sufficiently large constant $c'$, proving the result.\qed
	
	\medskip
	
	It remains to prove Lemma \ref{lem:randTech}, and for ease of presentation we do this in the following two subsections.
	
	\subsection{Balanced Supersaturation}
	To adapt the proof of Theorem~6.1 of Morris and Saxton~\cite{MS}, we require a balanced supersaturation result from Corsten and Tran~\cite{CT} which roughly says that if $H$ is an $r$-graph with significantly more than $n^{r-1/a_r}$ edges, then one can find a large collection $\c{H}$ of copies of $K_{s_1,\ldots,s_r}$ in $H$ which are relatively spread apart.
	
	More precisely, given an $r$-graph $H$, we identify copies of $K_{s_1,\ldots,s_r}$ in $H$ by ordered tuples $(S_1,\ldots,S_r)$ with $|S_i|=s_i$ such that these sets induce a  $K_{s_1,\ldots,s_r}$ in $H$.   If $\c{H}$ is a collection of copies of $K_{s_1,\ldots,s_r}$ in $H$ and $(T_1,\ldots,T_r)$ is an ordered tuple with $1\le |T_i|\le s_i$, we define $d_{\c{H}}(T_1,\ldots,T_r)$ to be the number of $(S_1,\ldots,S_r)\in \c{H}$ such that $T_i\sub S_i$ for all $i$.  If $t_1,\ldots,t_r$ are such that $1\le t_i\le s_i$ for all $i$, then for $\del,\ell,n>0$ we define the functions
	\[D^{t_1,\ldots,t_r}(\del,\ell,n)=\ell^{\sum_{i=1}^r a_i (s_i-t_i)}(\del n)^{\sum_{i=1}^r (s_i-t_i)} ,\]
	where we recall $a_i=\prod_{j<i} s_j$; and whenever $\del,\ell,n$ are understood we simply denote this function by $D^{t_1,\ldots,t_r}$.  For example, when $r=3$ we have
	\[D^{t_1,t_2,t_3}=\ell^{(s_1-t_1)+s_1(s_2-t_2)+s_1s_2(s_3-t_3)}(\del n)^{s_1-t_1+s_2-t_2+s_3-t_3}.\]
	We note that in \cite{CT} this function was defined in terms of $k=\ell n^{-1/a_r}$, but our intermediate computations will be greatly simplified by using this change in variables.  In particular, we can rephrase Theorem~3.1 of Corsten and Tran~\cite{CT} as follows.
	
	\begin{prop}[\cite{CT}] \label{prop:CT}
		For every $2\le s_1\le \cdots \le s_r$ there exist constants $\del,\ell_0>0$ such that the following holds for every $\ell \ge \ell_0 n^{-1/a_r}$ and every $n\in \N$.  Given an $r$-graph $H$ with $n$ vertices and $\ell n^{r}$ edges, there exists a collection $\c{H}$ of copies of $K_{s_1,\ldots,s_r}$ in $H$ such that
		\begin{enumerate}
			\item[$(a)$] $|\c{H}|\ge \del \ell^{a_{r+1}}n^{\sum_{i=1}^r s_i}$, and
			\item[$(b)$] $d_{\c{H}}(T_1,\ldots,T_r)\le D^{|T_1|,\ldots,|T_r|}$ for all $T_i\sub V(G)$ with $1\le |T_i|\le s_i$.
		\end{enumerate}
	\end{prop}
	We can treat $\c{H}$ from Proposition~\ref{prop:CT} as an $a_{r+1}$-uniform hypergraph with $V(\c{H})=E(H)$ and where edges $E\in E(\c{H})$ correspond to copies of $K_{s_1,\ldots,s_r}$ in $H$ which use all of the edges $e\in E$.  For any $\sig \sub E(H)$, let $d_{\c{H}}(\sig)$ denote the number of edges in $\c{H}$ containing $\sig$.  To apply the method of hypergraph containers, we show that $d_{\c{H}}(\sig)$ is relatively small for all $\sig$.
	
	For $\sig\sub V(\c{H})=E(H)$, define $V(\sig):=\bigcup_{e\in \sig}e$ to be the vertices involving edges of $\sig$.  Observe that if $(S_1,\ldots,S_r)$ corresponds to some $K_{s_1,\ldots,s_r}$ in $\c{H}$ containing $\sig$, then this copy contributes to $d_{\c{H}}(T_1,\ldots,T_r)$ where $T_i=S_i\cap V(\sig)$. Thus if $\c{P}$ is the set of all $r$-partitions of $V(\sig)$, we have \begin{equation}\label{eq:deg}d_{\c{H}}(\sig)\le \sum_{(T_1,\ldots,T_r)\in \c{P}} d_{\c{H}}(T_1,\ldots,T_r)= O(\max D^{t_1,\ldots,t_r}),\end{equation}
	where the maximum ranges over all integers $1\le t_i\le s_i$ with $\prod t_i\ge |\sig|$ (because the $r$-graph induced by any element of $\c{P}$ must be a complete $r$-partite $r$-graph to have positive degree in $\c{H}$).  To simplify computations, we extend the definition of $D^{t_1,\ldots,t_r}$ and the maximum of \eqref{eq:deg} to all real $t_i$ in this range.  Observe that $D^{t_1,\ldots,t_r}$ is a decreasing function in each of the variables $t_i$ provided $\ell^{a_r}\ge (\del n)^{-1}$, so possibly by setting $\ell_0=\del^{-1/a_r}$ we can assume for our range of $\ell$ that we have $\prod t_i=|\sig|$ exactly.
	
	For ease of notation we define \[ b_i=\prod_{j>i} s_j,\]
	where we adopt the convention $b_r=1$.  Because $\ell=O(1)$ (since the $r$-graph $H$ has at most $O(n^r)$ edges), we see that $D^{t_1,\ldots,t_r}$ is maximized given $\prod t_i=|\sig|$ when $t_r$ is made as large as possible, with $t_{r-1}$ as large as possible subject to this, and so on. In particular, if $|\sig|$ lies in the interval $[b_i,b_{i-1}]$, then the maximum of \eqref{eq:deg} occurs when $t_j=s_j$ for all $j>i$, $t_j=1$ for all $j<i$, and $t_i=|\sig|/b_j$.  Thus
	\begin{equation*}d_{\c{H}}(\sig)=O\l(\ell^{\sum_{j<i} a_j(s_j-1)+a_i(s_i-|\sig|/b_i)} (\del n)^{\sum_{j\le i}s_j-i+1-|\sig|/b_i}\r)\tr{ for }|\sig|\in [b_i,b_{i-1}],
	\end{equation*}
	and because $a_js_j=a_{j+1}$ for all $j$, this can be written more succinctly as
	\begin{align}
	\label{eq:interval}d_{\c{H}}(\sig)&=O\l(\ell^{a_{i}(s_i+1-|\sig|/b_i)-1}(\del n)^{\sum_{j\le i}s_j-i+1-|\sig|/b_i}\r)\tr{ for }|\sig|\in [b_i,b_{i-1}],\\
	\label{eq:interval2}d_{\c{H}}(\sig)&=O\l(\ell^{a_{i+1}-1}(\del n)^{\sum_{j\le i}s_j-i}\r)\tr{ for }|\sig|=b_i.
	\end{align}
	In particular, \eqref{eq:interval2} implies
	\begin{equation}\label{eq:Del1} d_{\c{H}}(\sig)=O\l(\ell^{a_{r+1}-1}(\del n)^{\sum_{i=1}^r s_i-r}\r),\tr{ for }|\sig|=1.\end{equation}
	
	From now on we consider $\sig$ with $|\sig|\ge2$ and define
	\[\phi(\sig)=\l(\f{e(H)}{|\c{H}|}d_{\c{H}}(\sig)\r)^{\rec{|\sig|-1}},\hspace{2em} \tau'=\max_{\sig:\ 2\le|\sig|\le a_{r+1}} \phi(\sig).\]
	Using \eqref{eq:interval} and the bound on $|\c{H}|$ from Proposition~\ref{prop:CT}, we can bound $\phi(\sig)$ by a function of the form
	\[O\l(\ell^{\f{c_i+d_i|\sig|}{|\sig|-1}}(\del n)^{\f{c_i'+d_i'|\sig|}{|\sig|-1}}\r),\]
	where the constants $c_i,d_i,c_i',d_i'$ depend on which $[b_i,b_{i-1}]$ interval $|\sig|$ belongs to.  With this formulation and a bit of calculus, one sees that restricted to any $[b_i,b_{i-1}]$ interval this bound on $\phi(\sig)$ is either non-decreasing or non-increasing (depending only on the value of $\ell$ relative to $\del n$).  In particular, to upper bound $\tau'$ it is enough to use these upper bounds for $\phi(\sig)$ whenever $|\sig|$ is an endpoint of some $[b_i,b_{i-1}]$ interval.
	
	To this end, first observe that for all $i$, \[e(H)/|\c{H}|=O\l(  \ell^{1-a_{r+1}}(\del n)^{r-\sum_{j=1}^r s_j}\r)=O\l(\ell^{1-a_{i+1}b_{i}}(\del n)^{r-\sum_{j=1}^r s_j}\r).\]  Thus for $0\le i<r$ we have by \eqref{eq:interval2} that with $|\sig|=b_i$,
	\[\phi(\sig)=O\l( \l(\ell^{a_{i+1}(1-b_i)}(\del n)^{r-i-\sum_{j>i}s_j}\r)^{1/(b_i-1)}\r)=O\l(\ell^{-a_{i+1}}(\del n)^\f{r-i-\sum_{j>i}s_j}{b_i-1}\r),\]
	and for all relevant $|\sig|\in [b_r,b_{r-1}]$ (i.e. those with $2\le |\sig|\le b_{r-1}=s_r$), we have by \eqref{eq:interval} that
	\[\phi(\sig)=O\l( (\ell^{a_r(1-|\sig|)}(\del n)^{1-|\sig|})^{1/(|\sig|-1)}\r)=O(\ell^{-a_r}(\del n)^{-1}).\]
	Putting all of this together, we find
	\begin{equation}\label{eq:tau}\tau' = O\l(\max_{0\le i\le r-1}\left\{\ell^{-a_{i+1}}(\del n)^\f{r-i-\sum_{j>i}s_j}{b_i-1}\right\}\r).\end{equation}
	
	We claim that only the $i=0,r-1$ terms of this maximum are relevant.  Indeed, observe that the point at which the $(r-1)$st term $\ell^{-a_r}(\del n)^{-1}$ equals the $i$th term is exactly when
	\[\log_{\del n}(\ell)= \f{\sum_{j>i}s_j-(r-i)-(b_i-1)}{(a_r-a_{i+1})(b_i-1)}=\f{\sum_{j>i}s_j-(r-i-1)-b_i}{a_{i+1}(b_{i}s_r^{-1}-1)(b_i-1)}:=\gam_i\]
	Because $b_i=b_{i+1}s_{i+1}$ and $s_{i+1}\ge 1$, we have \begin{align*}b_{i}-1&\ge s_{i+1}(b_{i+1}-1),\\ b_{i}s_r^{-1}-1&\ge s_{i+1}(b_{i+1}s_r^{-1}-1).\end{align*}  We also have \[\sum_{j>i} s_j-(r-i-1)=s_{i+1}+\sum_{j>i+1} s_j-(r-i-1)\le s_{i+1}+s_{i+1}\l(\sum_{j>i+1} s_j-(r-i-2)\r),\]
	where this last step used $s_{i+1}\ge 1$ and $\sum_{j>i+1} s_j\ge (r-i-2)$ since $s_j\ge 1$ for all $j$.  These observations imply
	\begin{align*}
	\gam_{i}\le \f{s_{i+1}+s_{i+1}(\sum_{j>i+1}s_j-(r-i-2))-b_{i}}{a_{i+1}\cdot s_{i+1}^2(b_{i+1}s_r^{-1}-1)(b_{i+1}-1)}=\f{1+\sum_{j>i+1} s_j-(r-i-2)-b_{i+1}}{a_{i+2}(b_{i+1}s_r^{-1}-1)(b_{i+1}-1)}=\gam_{i+1},
	\end{align*}
	where this first equality used $a_{i+1}s_{i+1}=a_{i+2}$ and $b_{i}=b_{i+1}s_{i+1}$. In total, for $\ell\le n^{\gam_{0}}$ the $i=r-1$ term of \eqref{eq:tau} is the maximum, and at $\ell=n^{\gam_{0}}$ the $i=0$ term is equal to the $i=r-1$ term.  Because the $i=0$ term has the largest power of $\ell$, it will continue to be the maximum value for all $\ell\ge n^{\gam_{0}}$, proving the claim.
	
	Thus $\tau'=O\l(  \max\{\ell^{-a_r}(\del n)^{-1},\ell^{-1}(\del n)^{(r-\sum_{i=1}^rs_i)/(b_0-1)}\}\r)$.  If we let $k=\ell (\del n)^{-1/a_r}$, then note that the exponent of $\del n$ in this second term of the maximum equals
	\begin{equation}\label{eq:gam}a_r^{-1}-\f{\sum_{i=1}^{r} s_i-r}{b_0-1}=-\f{a_r(\sum_{i=1}^{r-1}s_i-r)+1}{a_r(a_{r+1}-1)}:=-\gam.\end{equation}
	Using this and \eqref{eq:Del1}, we can reformulate Proposition~\ref{prop:CT} as follows.
	
	\begin{prop}\label{prop:CT2}
		For every $2\le s_1\le \cdots \le s_r$ there exist constants $\del,k_0>0$ such that the following holds for every $k \ge k_0$ and every $n\in \N$.  Given an $r$-graph $H$ with $n$ vertices and $k n^{r-1/a_r}$ edges, there exists a collection $\c{H}$ of copies of $K_{s_1,\ldots,s_r}$ in $H$ which we view as an $a_{r+1}$-graph such that
		\begin{enumerate}
			\item[$(a)$] $e(\c{H})\ge \del k^{a_{r+1}}n^{\sum_{i=1}^{r-1} s_i}$, and
			\item[$(b)$]  If $\Del_j$ is the maximum $j$-degree of $\c{H}$, then
			\[\Del_1\le k^{a_{r+1}-1} (\del n)^{ \sum_{i=1}^{r-1}s_i-r+1/a_r},\]
			
			and for $j\ge 2$ we have
			\[\l(\f{e(H)}{e(\c{H})}\Del_j\r)^{1/(j-1)}=O(\del^{-\gam} \max\{k^{-a_r},k^{-1}n^{-\gam}\}).\]
		\end{enumerate}
	\end{prop}
	
	\subsection{Proof of Lemma \ref{lem:randTech}}
	
	With Proposition~\ref{prop:CT2} established, the other technical tool we need to prove Lemma \ref{lem:randTech} is the following container lemma.
	
	\begin{lem}[\cite{MS}]\label{lem:containers}
		Let $q\ge 2$ and $0<\del<\del_0(q)$ be sufficiently small.  Let $\c{H}$ be a $q$-graph with $N$ vertices and maximum $j$-degree $\Del_j$ for all $j$, and suppose $\tau>0$ is such that   \[\f{|V(\c{H})|}{e(\c{H})}\sum_{j=2}^{q} \f{\Del_j}{\tau^{j-1}}\le \del.\]
		Then there exists a collection $\c{C}$ of subsets of $V(\c{H})$ and a function $f$ from subsets of $V(\c{H})$ to $\c{C}$ such that:
		\begin{itemize}
			\item[$(a)$] for every independent set $I\sub V(\c{H})$ there exists $T\sub I$ with $|T|\le \tau N/\del$ and $I\sub f(T)$, and
			\item[$(b)$] $e(\c{H}[C])\le (1-\del)e(\c{H})$ for every $C\in \c{C}$.
		\end{itemize}
	\end{lem}
	
	With these two results we prove the following.
	
	\begin{lem}\label{lem:containersStrong}
		For every $2\le s_2\le \cdots\le s_r$, there exist $\ep,k_0>0$ such that the following holds for every $k\ge k_0$ and $n\in \N$.  Set $\mu=\ep^{-1}k^{-1}\max \{k^{1-a_r},n^{-\gam}\}$.  Given an $r$-graph $H$ with $n$ vertices and $kn^{r-1/a_r}$ edges, there exists a function $f_H$ that maps subgraphs of $H$ to subgraphs of $H$ such that for every $K_{s_1,\ldots,s_r}$-free subgraph $I\sub H$ we have:
		\begin{itemize}
			\item[$(a)$] There exists a subgraph $T\sub I$ with $e(T)\le \mu n^{r-1/a_r}$, and
			\item[$(b)$] $I\sub f_H(T)$ with $e(f_H(T))\le (1-\ep) e(H)$.
		\end{itemize}
	\end{lem}
	\begin{proof}
		Let $\del,\ k_0$, and $\c{H}$ be as in Proposition~\ref{prop:CT2}, and assume that $\del$ is sufficiently small so that Lemma~\ref{lem:containers} applies with $q=a_{r+1}$ (otherwise we can take a smaller $\del$ and the result of Proposition~\ref{prop:CT2} continues to hold).
		Observe that $\tau=\del^{-2-\gam} \max\{k^{-a_r},k^{-1}n^{-\gam}\}$ gives
		\[\f{|V(\c{H})|}{e(\c{H})}\sum_{j=2}^{q} \f{\Del_j}{\tau^{j-1}}=O(\del^2),\]
		and by taking $\del$ sufficiently  small we can assume this sum is at most $\del$.
		
		Let $c=\max\{3+\gam,\sum_{i=1}^{r-1}s_i-r+1/a_r\}$ and $\ep=\del^c$.  By applying Lemma~\ref{lem:containers}, we obtain a collection $\c{C}$ of subsets of $V(\c{H})$ and a function $f_H$ from subsets of $V(\c{H})$ to $\c{C}$ such that for every $K_{s_1,\ldots,s_r}$-free subgraph $I\sub H$ we have:
		\begin{itemize}
			\item[$(a)$] There exists a subgraph $T\sub I$ with \[e(T)\le (\del^{-2-\gam}\max\{k^{-a_r},k^{-1}n^{-\gam}\})(k n^{r-1/a_r})/\del \le \mu n^{r-1/a_r},\] and
			\item[$(b')$] $I\sub f_H(T)$ with $e(f_H(T))\le (1-\del)e(\c{H})$.
		\end{itemize}
		To show (b), it suffices to show that (b') implies $e(C)\le (1-\ep)e(H)$ for every $C\in \c{C}$.  Indeed, let
		\[\c{D}(C)=E(\c{H})\sm E(\c{H}[C])=\{E\in E(\c{H}) :e\in E\tr{ for some }e\in E(H)\sm C\}.\]
		By definition $|\c{D}(C)|=e(\c{H})-e(\c{H}[C])$, and this is at least $\del e(\c{H})$ by (b'). By the bound on the maximum degree of $\c{H}$ from  Proposition~\ref{prop:CT2}, we find
		\[|\c{D}(C)|\le \f{e(\c{H})}{\del^{c-1} k n^{r-1/a_r}}\cdot |E(H)\sm C|.\]
		Combining these two results implies $|E(H)\sm C|\ge \ep kn^{r-1/a_r}$ as desired.
	\end{proof}

	Lastly, we need the following inequality.
	\begin{lem}[\cite{MN}]\label{lem:bound}
		Let $M,s>0$, and $0<\del<1$.  If $b_1,\ldots,b_m\in \R$ satisfy $s=\sum b_j$ and $1\le b_j\le (1-\del)^j M$ for each $j\in [m]$, then
		\[s\log s\le \sum b_j \log b_j+O(M).\]
	\end{lem}
	\begin{proof}[Proof of Lemma~\ref{lem:randTech}]
		We construct the functions $g,h$ and family $\c{S}$ as follows.  Given a $K_{s_1,\ldots,s_r}$-free $r$-graph $I\in \c{I}$, we repeatedly apply Lemma~\ref{lem:containersStrong} first to $H_0=K_n^{r}$, then to $H_1=f_{H_0}(T_1)\sm T_1$, where $T_1\sub I$ is the set guaranteed to exist by (a) of Lemma~\ref{lem:containersStrong}; then to $H_2=f_{H_1}(T_2)\sm T_2$ where $T_2\sub I\cap H_1=I\sm T_1$, and so on.  We continue until we arrive at an $r$-graph $H_m$ with at most $k n^{r-1/a_r}$ edges and set $g(I)=(T_1,\ldots,T_m)$ and $h(g(I))=H_m$.  Since $H_m$ depends only on the sequence $(T_1,\ldots,T_m)$, the function $h$ is well-defined.
		
		It remains to bound the number of colored graphs in $\c{S}$ with $s$ edges. To do this, it suffices to count the number of choices for the sequence of $r$-graphs $(T_1,\ldots,T_m)$ with $\sum e(T_j)=s$.  For each $j\ge 1$, define $k(j)$ and $\mu(j)$ by $e(H_{m-j})=k(j)n^{r-1/a_r}$ and $\mu(j)=\ep^{-1} \max \{k(j)^{1-a_r},n^{-\gam}\}$, and note that \[ (1-\ep)^{-j+1}k\le k(j)=O( n^{1/a_r}),\ T_{j+1}\sub H_j,\ e(T_{m-j})\le \mu(j)n^{r-1/a_r}.\]
		
		Thus fixing $k,\ep,s$ as above, we define \begin{align*}\c{K}(m)&=\{\b{k}=(k(1),\ldots,k(m)):(1-\ep)^{-j+1}k\le k(j)\le n^{1/a_r}\}\textrm{ for each }m\in \N,\\
		\c{B}(\b{k})&=\{\b{b}=(b(1),\ldots,b(m)):b(j)\le\mu(j)n^{r-1/a_r}\tr{ and }\sum b(j)=s\}\textrm{ for each }\b{k}\in \c{K}(m).\end{align*}  It follows that the number of colored graphs in $\c{S}$ with $s$ edges is at most
		\[\sum_{m=1}^\infty \sum_{\b{k}\in \c{K}(m)}\sum_{\b{b}\in \c{B}(\b{k})}\prod_{j=1}^m{k(j)n^{r-1/a_r}\choose b(j)}.\]
		
		Given $m,\b{k},\b{b}$, we partition this product over $j$ according to whether or not $\mu(j)=\ep^{-1} n^{-\gam}$.  Since $\c{K}(m)=\emptyset$ for $m=\Om( \log n)$ and $b(j)\le \mu(j)n^{r-1/a_r}\le \ep^{-1}n^{r-1/a_r-\gam}$ in this case, the product over these terms is at most
		\[(n^r)^{\sum b(j)}\le \exp(O(1)\cdot n^{r-1/a_r-\gam}(\log n)^2)\le \exp(O(1)\cdot k^{1-a_r} n^{r-1/a_r}),\]
		where this last step used the hypothesis of $k\le n^{\be_2/a_r} (\log n)^{2/(a_r-1)}$ and that $\be_2/a_r=\gam/(a_r-1)$ by the way we defined $\gam$ in \eqref{eq:gam}.  On the other hand, if $b(j)\le \ep^{-1} k(j)^{1-a_r}n^{r-1/a_r}$, then
		\[{k(j) n^{r-1/a_r}\choose b(j)}\le \l(\f{ek(j)n^{r-1/a_r}}{ b(j)}\r)^{b(j)}\le \l(\f{n^{r-1/a_r}}{\ep b(j)}\r)^{a_r b(j)/(a_r-1)}.\]
		
		Thus by Lemma~\ref{lem:bound} the product of these remaining $j$ terms is at most
		\[\l(\f{c n^{r-1/a_r}}{s}\r)^{a_r s/(a_r-1)}\exp(ck^{1-a_r} n^{r-1/a_r})\]
		for some constant $c$.  The result follows after noting $\sum_m \sum_{\b{k}} |\c{B}(\b{k})|=n^{O(\log n)}$.
	\end{proof}
	
	\section{Concluding Remarks}\label{sec:concluding}
	$\bullet$  Foucaud, Krivelevich, and Perarnau~\cite{FKP} conjectured that if $F$ and $H$ are graphs such that $H$ has minimum degree $\del$ and maximum degree $\Del$, then $H$ has a spanning $F$-free subgraph of minimum degree $\Omega(\del \ex(\Del,F)/\Del^2)$ as $\Del \rightarrow \infty$. This conjecture was proved for bipartite graphs of diameter at most three by Perarnau and Reed~\cite{PR}. A key part of the proof is to show that for some $c,c'> 0$, every $\Del$-regular graph $H$ has a spanning subgraph $G$ of minimum degree at least $c\Del$ with an \emph{injective $c'\Del$-coloring}, which is a map $\chi : V(G) \rightarrow \{1,2,\dots,c'\Del\}$ such that every pair of edges $e_1,e_2$ with $|e_1\cap e_2|=1$ has $\chi(e_1)\ne \chi(e_2)$.  It is natural to consider a similar framework for $r$-graphs, where now we require that $e_1,e_2$ with $|e_1\cap e_2|=r-1$ have distinct color sets.

	\begin{conj} \label{conj:injective}
		There exist constants $c,c'>0$ such that if $H$ is a $\Del$-regular $r$-graph with maximum $(r-1)$-degree at most $D$, then $H$ has a spanning subgraph $G$ of minimum degree at least $c\Del$ with an
		injective $c'D$-coloring.
	\end{conj}
	Note that the proof of Lemma~\ref{lem:homGen} essentially shows that one can find a subgraph of $H$ with at least $\Om(e(H))$ edges which has an injective $O(D)$-coloring, so the central difficulty is in maintaining the minimum degree.
	
	\medskip
	
	$\bullet$ The main open question for relative Tur\'{a}n numbers of $r$-graphs is to give bounds on $\ex(H,F)$ when $F$ is $r$-partite. We observe that for each positive integer $\Delta$ there exists an $r$-graph $H$ of maximum degree at most $\Delta$ such that
	\begin{equation}\label{suggested}
	\ex(H,F) = O\Bigl(\frac{\ex(\Del^{\frac{1}{r - 1}},F)}{\Del^{\frac{r}{r - 1}}}\Bigr) \cdot e(H),
	\end{equation}
	namely with $H=K_{\Del^{1/(r-1)}}^r$.  This leads to the question of determining for which $F$ the above upper bound is tight up to constants for all $r$-graphs $H$ of maximum degree $\Delta$ --
	Conjecture \ref{conj:graphHost} states that this holds for all graphs $F$. To this end, we generalize \eqref{Turanexp} by defining the {\em Tur\'{a}n exponent} of an $r$-graph $F$, when it exists,
	to be
	\[ \alpha(F) = \lim_{n \rightarrow \infty} \frac{\log {n \choose r}/\ex(n,F)}{\log {n-1\choose r-1}}.\]
	It seems likely that $\alpha(F)$ exists for every $r$-graph $F$, and the existence of $\alpha(F)$ when $F$ is a graph is a consequence of a conjecture of Erd\H{o}s and
	Simonovits~\cite{ES}. Similarly we generalize \eqref{genTuranexp} by defining the {\em relative Tur\'{a}n exponent} of $F$, when it exists, to be
	\[ \beta(F) = \lim_{\Delta \rightarrow \infty}\sup_{H} \frac{\log e(H)/\ex(H,F)}{\log \Delta},\]
	where the supremum ranges over all $H$ with maximum degree at most $\Del$. Theorem \ref{thm:K2} shows that whenever each $K_{s_1,\ldots,s_i}$ is known to have Tur\'{a}n exponent $(s_1\cdots s_{i-1})^{-1}$, the relative Tur\'{a}n exponent
	exists and is given by (\ref{genTuranexp}). It is noteworthy that unlike $r = 2$, $ \alpha(K_{s_1,\ldots,s_r})<\be(K_{s_1,\ldots,s_r})$ for $r \geq 3$.
	
	\medskip
	
	$\bullet$ Analogous to the conjecture that $\al(F)$ exists for all graphs, we conjecture $\be(F)$ exists for all $F$.
	
	\begin{conj}\label{general}
		For every $r$-graph $F$, the relative Tur\'{a}n exponent $\beta(F)$ exists.
	\end{conj}
	
	For $r = 2$, if Conjecture~\ref{conj:graphHost} were true then the existence of $\beta(F)$ would follow from the existence of $\alpha(F)$, and in fact $\alpha(F) = \beta(F)$ in this case. While we have $\alpha(F) \leq \beta(F)$ for all $r$-graphs $F$, by (\ref{suggested}) we see that these quantities may differ sharply in the setting of hypergraphs.  For example, Theorem \ref{thm:K2} shows $\alpha(F)<\beta(F)$ when $F = K_{s_1,s_2,\dots,s_r}$ and $r \geq 3$.  It seems difficult in general to determine whether $\al(F)=\be(F)$ for a given $F$, and we leave this as an open problem.
	
	\begin{prob}
		Determine the $r$-partite $r$-graphs for which $\al(F)=\be(F)$.
	\end{prob}

	\medskip
	
	$\bullet$ In Theorem~\ref{thm:K2} we determined $\be(F)$ for $F=K_{s_1,\ldots,s_r}$ and certain values of $s_i$, and in this case we showed $\al(F)\ne \be(F)$.  Our proof extends to a somewhat wider family of hypergraphs as follows.
	
	Given a graph $F$, we define its $s$-extension to be the $3$-graph $F_{+s}$ on $V(F)\cup [s]$ with edge set $E(F_{+s})=\{ e\cup \{i\}:e\in E(F),\ i\in[s]\}$.  For example, $(K_{s_1,s_2})_{+s_3}=K_{s_1,s_2,s_3}^{(3)}$.   By going through a nearly identical proof as that of Theorem~\ref{thm:K2} and using the method of random polynomials due to Bukh and Conlon~\cite{BC}, it is possible to determine $\be(F_{+s})$ for $s$ sufficiently large provided $F$ is a non-empty connected bipartite graph of diameter at most 3 which has a supersaturation result analogous to the result of Erd\H{o}s and Simonovits \cite{ES} that was used in the proof of Lemma~\ref{lem:construct}.  In this setting we further have that $\al(F_{+s})\ne \be(F_{+s})$.
	\bibliographystyle{abbrv}
	\bibliography{GT}
\end{document}